\documentclass[oneside,english]{amsart}
\usepackage[T1]{fontenc}
\usepackage[latin9]{inputenc}
\usepackage{geometry}
\usepackage{color}
\usepackage{amstext}
\usepackage{amsthm}
\usepackage{amssymb}

\makeatletter
\numberwithin{equation}{section}
\numberwithin{figure}{section}
\theoremstyle{definition}
\newtheorem*{example*}{\protect\examplename}
\theoremstyle{plain}
\newtheorem{thm}{\protect\theoremname}
\theoremstyle{plain}
\newtheorem{lem}[thm]{\protect\lemmaname}
\theoremstyle{definition}
\newtheorem{defn}[thm]{\protect\definitionname}
\theoremstyle{plain}
\newtheorem{prop}[thm]{\protect\propositionname}
\theoremstyle{remark}
\newtheorem{rem}[thm]{\protect\remarkname}
\theoremstyle{definition}
\newtheorem{example}[thm]{\protect\examplename}
\theoremstyle{plain}
\newtheorem{cor}[thm]{\protect\corollaryname}

\usepackage{xypic}
\usepackage{graphicx}
\usepackage{tikz}
\usepackage{tikz-cd}
\usepackage{mathrsfs, setspace, fancyhdr, times, bm}
\DeclareFontFamily{OT1}{pzc}{}
\DeclareFontShape{OT1}{pzc}{m}{it}%
             {<-> s * [1.195] pzcmi7t}{}
\DeclareMathAlphabet{\mathscr}{OT1}{pzc}%
                                 {m}{it}
\usepackage[colorlinks=true]{hyperref} 
\hypersetup{linkcolor  = blue}

\renewcommand{\upsilon}{\theta}

\makeatother

\usepackage{babel}
\providecommand{\corollaryname}{Corollary}
\providecommand{\definitionname}{Definition}
\providecommand{\examplename}{Example}
\providecommand{\lemmaname}{Lemma}
\providecommand{\propositionname}{Proposition}
\providecommand{\remarkname}{Remark}
\providecommand{\theoremname}{Theorem}

\begin{document}
\author{Adrien Dubouloz}
\address{IMB UMR5584, CNRS, Univ. Bourgogne Franche-Comté, F-21000 Dijon, France.}
\email{adrien.dubouloz@u-bourgogne.fr}
\author{Isac Hedén}
\address{Mathematics Institute, Zeeman Building, University of Warwick Coventry
CV4 7AL, UK. }
\email{I.Heden@warwick.ac.uk}
\author{Takashi Kishimoto }
\address{Department of Mathematics, Faculty of Science, Saitama University,
Saitama 338-8570, Japan }
\email{tkishimo@rimath.saitama-u.ac.jp}
\thanks{The first author was partially supported by the French \textquotedbl Investissements
d\textquoteright Avenir\textquotedbl{} program, project ISITE-BFC
(contract ANR-lS-IDEX-OOOB) and from ANR Project FIBALGA ANR-18-CE40-0003-01.
The second author gratefully acknowledges support from the Knut and
Alice Wallenberg Foundation, grant number KAW2016.0438. The third
author was partially funded by Grant-in-Aid for Scientific Research
of JSPS No. 15K04805 and No. 19K03395. The authors thank the University
of Saitama, at which this research was initiated during visits of
the first and second authors, and the Institute of Mathematics of
Burgundy, at which it was continued during a visit of the third author,
for their generous support and the excellent working conditions offered.}
\title{Rees algebras of additive group actions }
\begin{abstract}
We establish basic properties of a sheaf of graded algebras canonically
associated to every relative affine scheme $f:X\rightarrow S$ endowed
with an action of the additive group scheme $\mathbb{G}_{a,S}$ over
a base scheme or algebraic space $S$, which we call the (relative)
Rees algebra of the $\mathbb{G}_{a,S}$-action. We illustrate these
properties on several examples which played important roles in the
development of the algebraic theory of locally nilpotent derivations
and give some applications to the construction of families of affine
threefolds with $\mathbb{G}_{a}$-actions.
\end{abstract}

\maketitle

\section*{Introduction}

The study of regular actions of the additive group $\mathbb{G}_{a}$
on affine varieties has led to an increased understanding of both
algebraic and geometric properties of these varieties. A $\mathbb{G}_{a}$-action
on an affine variety $X$ defined over a field $k$ of characteristic
zero is fully determined by its velocity vector field, which takes
the form of a $k$-derivation $\partial$ of the coordinate ring $A$
of $X$ with the property that $A$ is the increasing union of the
kernels of the iterated $k$-linear differential operators $\partial^{n}$,
$n\geq1$. Due to this correspondence, the study of regular $\mathbb{G}_{a}$-actions
developed into a very rich algebraic theory of such differential operators,
called \emph{locally nilpotent} $k$-\emph{derivations. }The most
fundamental object associated to such a derivation is its kernel,
which coincides with the subalgebra of $\mathbb{G}_{a}$-invariant
functions on $X$. Kernels of locally nilpotent derivations have been
intensively studied during the last decades, with many applications
to the construction of new invariants to distinguish affine spaces
among all affine varieties and to the understanding of automorphism
groups of affine varieties close to affine spaces (see \cite{FreuBook}
and the references therein). A second natural subspace associated
to a locally nilpotent $k$-derivation $\partial$ of a $k$-algebra
$A$ which has been very much studied from an algebraic point of view
is the kernel $\mathrm{Ker}\partial^{2}$ of its square $\partial^{2}$.
Geometrically, the elements of $\mathrm{Ker}\partial^{2}\setminus\mathrm{Ker}\partial$,
usually called local slices, are regular functions on $X=\mathrm{Spec}(A)$
which restrict to coordinate functions on general orbits of the corresponding
$\mathbb{G}_{a}$-action. The image of $\partial|_{\mathrm{Ker}\partial^{2}}:\mathrm{Ker}\partial^{2}\rightarrow\mathrm{Ker}\partial$
is an ideal of $\mathrm{Ker}\partial$, called the plinth ideal, which
encodes basic geometric properties of the algebraic quotient morphism
$X\rightarrow X/\!/\mathbb{G}_{a}=\mathrm{Spec}(\mathrm{Ker}\partial)$.

A more systematic study of the algebro-geometric properties encoded
by the whole increasing exhaustive filtration of $A$ formed by the
subspaces $F_{n}=\mathrm{Ker}\partial^{n+1}$, $n\geq0$, was initiated
only quite recently by Alhajjar \cite{AlH05,AlH05-2} and Freudenburg
\cite{Fr16}. For instance, they observed that for an integral finitely
generated $k$-algebra $A$ endowed with a nonzero locally nilpotent
$k$-derivation $\partial$, the infinite collection of inclusions
$F_{n}\hookrightarrow F_{n+1}$, $n\geq0$, gives rise to a collection
of successive inclusions $k[F_{n}]\hookrightarrow k[F_{n+1}]$ between
the $k$-subalgebras of $A$ that they generate. This sequence exhausts
$A$ after finitely many steps, i.e. $k[F_{r}]=A$ for some $r\in\mathbb{N}$.
These inclusions correspond geometrically to a canonical sequence
of birational $\mathbb{G}_{a}$-equivariant morphisms 
\[
X=\mathrm{Spec}(A)\cong\mathrm{Spec}(k[F_{r}])\rightarrow\mathrm{Spec}(k[F_{r-1}])\rightarrow\cdots\rightarrow\mathrm{Spec}(k[F_{1}])
\]
factorizing the algebraic quotient morphism $X\rightarrow\mathrm{Spec}(F_{0})$.
The basic properties of this factorization have been established by
Freudenburg \cite{Fr16}. He described in particular an algorithm
to compute, under suitable noetherianity conditions, the subspaces
$F_{n}$, $n\geq0$, as well as the corresponding algebras $k[F_{n}]$.

In this article, we shift the focus to a complementary approach which
considers the properties of the two natural graded algebras that can
be canonically associated to a filtered algebra: its associated graded
algebra and its Rees algebra. For a $k$-algebra $A$ with a nonzero
locally nilpotent $k$-derivation $\partial$, these are thus the
algebras 
\[
\mathrm{gr}_{\partial}A=\bigoplus_{n\geq0}\mathrm{Ker}\partial^{n+1}/\mathrm{Ker}\partial^{n}\quad\textrm{ and }\quad R(A,\partial)=\bigoplus_{n\geq0}\mathrm{Ker}\partial^{n+1}.
\]
The first one already plays an important role in the computation of
Makar-Limanov invariants of certain affine varieties \cite{KML07,AlH05-2},
but to our knowledge, the second one, which we henceforth call the
\emph{Rees algebra of} $(A,\partial)$, has not been considered before
in this context. Besides the fact that these Rees algebras are functorial
with respect to $\mathbb{G}_{a}$-equivariant morphisms, two basic
properties which motivate their study are the following:

$\bullet$ First, the canonical graded homomorphism of degree $0$
\[
k[\upsilon]=\bigoplus_{n\geq0}k\rightarrow\bigoplus_{n\geq0}\mathrm{Ker}\partial^{n+1}
\]
induced by the inclusion $k\subset\mathrm{Ker}\partial^{n+1}$ for
every $n\geq0$ provides a one-parameter deformation 
\[
\pi:\mathrm{Spec}(R(A,\partial))\rightarrow\mathbb{A}_{k}^{1}=\mathrm{Spec}(k[\upsilon])
\]
whose fibers over the points $0$ and $1$ are canonically isomorphic
to $\mathrm{Spec}(\mathrm{gr}_{\partial}A)$ and $\mathrm{Spec}(A)$
respectively. 

$\bullet$ Second, the Rees algebra $R(A,\partial)$ carries canonical
extensions of $\partial$ to homogeneous locally nilpotent $k[\upsilon]$-derivations
of homogeneous degree $m$ for every $m\geq-1$, whose corresponding
$\mathbb{G}_{a}$-actions on $\mathrm{Spec}(R(A,\partial))$ make
$\pi$ a $\mathbb{G}_{a}$-equivariant deformation of $\mathrm{Spec}(A)$
endowed with the $\mathbb{G}_{a}$-action defined by $\partial$.
For $m=-1$, the induced $\mathbb{G}_{a}$-action on the fiber $\pi^{-1}(0)$
coincides with the one which is defined by the homogeneous $k$-derivation
$\mathrm{gr}\partial$ of $\mathrm{gr}_{\partial}A$ of degree $-1$,
the latter derivation being canonically associated to $\partial$.
On the other hand, for $m=0$ the $\mathbb{G}_{a}$-action on $\mathrm{Spec}(R(A,\partial))$
descends to a $\mathbb{G}_{a}$-action on $\mathrm{Proj}(R(A,\partial))$
and we obtain in particular a canonical $\mathbb{G}_{a}$-equivariant
open embedding $\mathrm{Spec}(A)\hookrightarrow\mathrm{Proj}(R(A,\partial))$
which provides a canonical relative $\mathbb{G}_{a}$-equivariant
completion of $\mathrm{Spec}(A)$ over $\mathrm{Spec}(\mathrm{Ker}\partial)$.
\begin{example*}
The Rees algebra $R(k[t],\frac{\partial}{\partial t})$ of the additive
group $\mathbb{G}_{a}$ acting on itself by translations is isomorphic
to the polynomial ring in two variables $k[t,\upsilon]$ with its
standard grading, and the associated graded algebra $\mathrm{gr}_{\frac{\partial}{\partial t}}k[t]$
is the the polynomial ring $k[t]$ endowed with its standard grading.
The locally nilpotent $k[\upsilon]$-derivation $\frac{\partial}{\partial t}$
of $k[t,\upsilon]$ is homogeneous of degree $-1$, and $\mathrm{Spec}(R(k[t],\frac{\partial}{\partial t}))$
endowed with the corresponding $\mathbb{G}_{a}$-action is just the
trivial $\mathbb{G}_{a}$-torsor over $\mathrm{Spec}(k[\upsilon])$
via the projection 
\[
\pi=\mathrm{pr}_{\upsilon}:\mathrm{Spec}(R(k[t],\frac{\partial}{\partial t}))\rightarrow\mathrm{Spec}(k[\upsilon]).
\]
On the other hand, the locally nilpotent $k[\upsilon]$-derivation
$\upsilon\frac{\partial}{\partial t}$ of $k[t,\upsilon]$ is homogeneous
of degree $0$ and the open immersion 
\[
\mathbb{G}_{a}=\mathrm{Spec}(k[t])\hookrightarrow\mathbb{P}^{1}=\mathrm{Proj}_{k}(k[t,\upsilon]),\quad t\mapsto[t:1]
\]
is equivariant for the $\mathbb{G}_{a}$-action $t'\cdot[t:\upsilon]\mapsto[t+t'\upsilon:\upsilon]$
on $\mathbb{P}^{1}$ induced by $\upsilon\frac{\partial}{\partial t}$. 
\end{example*}
The content of the article is the following. In the first section
we establish basic general properties of Rees algebras of additive
group scheme actions in a relative and characteristic free setting.
Namely, given a fixed base scheme or algebraic space $S$, we consider
schemes or algebraic spaces $X$ that are endowed with an action of
the additive group scheme $\mathbb{G}_{a}$ and which admit a $\mathbb{G}_{a}$-invariant
affine morphism $f:X\rightarrow S$. Having this flexibility is useful
even in the absolute case of an affine variety $X$ over a base field
endowed with $\mathbb{G}_{a}$-action since we can then analyze the
relative structure of $X$ with respect to any $\mathbb{G}_{a}$-invariant
morphism $f:X\rightarrow S$ to some scheme or algebraic space. The
second section focuses on the case of algebraic varieties over a field
of characteristic zero. We study the behavior of Rees algebras under
certain type of equivariant morphisms and characterize geometrically
those that are finitely generated. We also describe an algorithm to
compute generators of these algebras. The last section is devoted
to a selection of examples which illustrate the interplay between
relative and absolute Rees algebras. We also present an application
of Rees algebras to the construction and classification of affine
extensions of $\mathbb{G}_{a}$-torsors over punctured surfaces \cite{DHK,He15},
a class of varieties which form one of the building blocks of the
classification theory of affine threefolds with $\mathbb{G}_{a}$-actions.

\hypersetup{linkcolor=black}

\tableofcontents{}

\hypersetup{linkcolor=blue}

\section{Rees algebras of affine $\mathbb{G}_{a}$-schemes}

Given a scheme or an algebraic space $S$, we denote by $\mathbb{G}_{a,S}=S\times_{\mathbb{Z}}\mathbb{G}_{a,\mathbb{Z}}=\mathrm{Spec}(\mathcal{O}_{S}[t])$
the additive group scheme over $S$. We denote by $m:\mathbb{G}_{a,S}\times_{S}\mathbb{G}_{a,S}\rightarrow\mathbb{G}_{a,S}$
and $e:S\rightarrow\mathbb{G}_{a,S}$ its group law and neutral section
respectively. By an affine $S$-scheme $f:X\rightarrow S$, we mean
the relative spectrum of a quasi-coherent sheaf $\mathcal{A}=f_{*}\mathcal{O}_{X}$
of $\mathcal{O}_{S}$-algebras. We say that $X$ is of finite type
over $S$ if $\mathcal{A}$ is locally of finite type as an $\mathcal{O}_{S}$-algebra. 

\subsection{Additive group scheme actions on relative affine schemes}

Let $S$ be a scheme or an algebraic space. An action $\mu:\mathbb{G}_{a,S}\times_{S}X\rightarrow X$
of $\mathbb{G}_{a,S}$ on an affine $S$-scheme $f:X=\mathrm{Spec}_{S}(\mathcal{A})\rightarrow S$
is equivalently determined by its $\mathcal{O}_{S}$-algebra co-morphism
\[
\mu^{*}:\mathcal{A}\rightarrow\mathcal{A}\otimes_{\mathcal{O}_{S}}\mathcal{O}_{S}[t]=\mathcal{A}[t],
\]
which satisfies the usual axioms of a group co-action, namely the
commutativity of the following two diagrams: \begin{equation}
\label{eq:Ga-action-commutative-diagram}
\xymatrix{\mathcal{A} \ar[r]^{\mu^*} \ar[d]_-{\mu^*} & \mathcal{A}\otimes_{\mathcal{O}_{S}}\mathcal{O}_{S}[t] \ar[d]^{\mathrm{id}\otimes m^{*}} & &   \mathcal{A} \ar[r]^-{\mu^*} \ar[dr]_{\mathrm{id}} & \mathcal{A}\otimes_{\mathcal{O}_{S}}\mathcal{O}_{S}[t] \ar[d]^{\mathrm{id}\otimes e^{*}} \\ \mathcal{A}\otimes_{\mathcal{O}_{S}}\mathcal{O}_{S}[t] \ar[r]^-{\mu^{*}\otimes\mathrm{id}} & \mathcal{A}\otimes_{\mathcal{O}_{S}}\mathcal{O}_{S}[t]\otimes_{\mathcal{O}_{S}}\mathcal{O}_{S}[t] & & & \mathcal{A}\otimes_{\mathcal{O}_{S}}\mathcal{O}_{S}\simeq\mathcal{A}.}
\end{equation}

For every $i\geq0$, let $\mathrm{p}_{i}:\mathcal{O}_{S}[t]=\bigoplus_{j=0}^{\infty}\mathcal{O}_{S}\rightarrow\mathcal{O}_{S}$
be the projection onto the $i$-th factor, and let $D^{(i)}=(\mathrm{id}_{\mathcal{A}}\otimes\mathrm{p}_{i})\circ\mu^{*}:\mathcal{A}\rightarrow\mathcal{A}$.
The following lemma is a well-known consequence of the commutativity
of the above diagrams.
\begin{lem}
\label{lem:OS-LFIHD-from-Ga}The $\mathcal{O}_{S}$-module endomorphisms
$D^{(i)}=(\mathrm{id}_{\mathcal{A}}\otimes\mathrm{p}_{i})\circ\mu^{*}:\mathcal{A}\rightarrow\mathcal{A}$
are differential operators of order $\leq i$ which satisfy the following
properties:

\begin{enumerate}

\item The operator $D^{(0)}$ is the identity map of $\mathcal{A}$,

\item For every $i\geq0$, the \emph{Leibniz rule} $D^{(i)}(ab)=\sum_{j=0}^{i}D^{(j)}(a)D^{(i-j)}(b)$
holds for every pair of local sections $a,b$ of $\mathcal{A}$ over
$S$,

\item For every $i,j\geq0$, $D^{(i)}\circ D^{(j)}=\binom{i+j}{i}D^{(i+j)}$,

\item We have $\mathcal{A}=\bigcup_{n\geq0}(\bigcap_{i>n}\mathcal{K}erD^{(i)})$.

\end{enumerate}
\end{lem}

\begin{proof}
The fact that $D^{(0)}=\mathrm{id}_{\mathcal{A}}$ follows from the
commutativity of the diagram on the right in \eqref{eq:Ga-action-commutative-diagram}.
Given local sections $a,b$ of $\mathcal{A}$ over $S$, the fact
that $\mu^{*}$ is a $\mathcal{O}_{S}$-algebra homomorphism implies
that $\mu^{*}(ab)=\mu^{*}(a)\mu^{*}(b)$. Writing $\mu^{*}(a)=\sum_{i\geq0}D^{(i)}(a)t^{i}$,
$\mu^{*}(b)=\sum_{i\geq0}D^{(i)}(b)t^{i}$ and $\mu^{*}(ab)=\sum_{i\geq0}D^{(i)}(ab)t^{i}$,
we have 
\[
D^{(i)}(ab)=\sum_{j=0}^{i}D^{(j)}(a)D^{(i-j)}(b)
\]
for every $i\geq0$, which proves Property (2). Let $a$ be a local
section of $\mathcal{A}$ over $S$ and write $\mu^{*}(a)=\sum_{i\geq0}D^{(i)}(a)t^{i}$
and for every $i\geq0$, $\mu^{*}(D^{(i)}(a))=\sum_{j\geq0}D^{(j)}(D^{(i)}(a))v^{j}$.
The commutativity of the diagram on the left in \eqref{eq:Ga-action-commutative-diagram}
implies that 
\[
\sum_{i\geq0}D^{(i)}(a)(t+v)^{i}=\sum_{i\geq0}(\sum_{j\geq0}D^{(j)}(D^{(i)}(a))v^{j})t^{i},
\]
from which it follows, by identifying the terms in $t^{i}v^{j}$,
that $D^{(j)}(D^{(i)}(a))=\binom{i+j}{i}D^{(i+j)}(a)$. This proves
Property (3). For every $n\geq0$, $\mathcal{F}_{n}=\bigcap_{i>n}\mathcal{K}erD^{(i)}$
is an $\mathcal{O}_{S}$-submodule of $\mathcal{A}$ which is equal
to the inverse image by $\mu^{*}$ of the $\mathcal{O}_{S}$-submodule
$\mathcal{A}\otimes_{\mathcal{O}_{S}}\mathcal{O}_{S}[t]_{\leq n}$
of $\mathcal{A}\otimes_{\mathcal{O}_{S}}\mathcal{O}_{S}[t]$. The
union $\bigcup_{n\geq0}\mathcal{F}_{n}$ is an $\mathcal{O}_{S}$-submodule
of $\mathcal{A}$ which is stable under multiplication by Property
(2), hence an $\mathcal{O}_{S}$-subalgebra of $\mathcal{A}$. The
fact that the inclusion $\bigcup_{n\geq0}\mathcal{F}_{n}\hookrightarrow\mathcal{A}$
is an isomorphism of $\mathcal{O}_{S}$-algebras is then checked on
an open cover of $S$ by affine open subsets as in \cite{Miy68}.
This proves Property (4).
\end{proof}
\begin{defn}
A collection of $\mathcal{O}_{S}$-linear differential operators $D^{(i)}:\mathcal{A}\rightarrow\mathcal{A}$,
$i\in\mathbb{Z}_{\geq0}$ which satisfy the properties of Lemma \ref{lem:OS-LFIHD-from-Ga}
is called a l\emph{ocally finite iterative higher $\mathcal{O}_{S}$}-\emph{derivation}
($\mathcal{O}_{S}$-LFIHD for short) of the quasi-coherent $\mathcal{O}_{S}$-algebra
$\mathcal{A}$.
\end{defn}

Conversely, for every $\mathcal{O}_{S}$-LFIHD $D=\{D^{(i)}\}_{i\geq0}$
of a quasi-coherent $\mathcal{O}_{S}$-algebra $\mathcal{A}$, the
exponential map 
\[
\mu^{*}:=\exp(tD)=\sum_{i=0}^{\infty}D^{(i)}\otimes t^{i}:\mathcal{A}\rightarrow\mathcal{A}\otimes_{\mathcal{O}_{S}}\mathcal{O}_{S}[t],\quad a\mapsto\sum_{i=0}^{\infty}D^{(i)}(a)\otimes t^{i}
\]
is the co-morphism of a $\mathbb{G}_{a,S}$-action $\mu:\mathbb{G}_{a,S}\times_{S}X\rightarrow X$
on $X$ (see e.g. \cite{Miy68}).

The quasi-coherent $\mathcal{O}_{S}$-algebra $\mathcal{A}$ of an
affine $S$-scheme $f:X=\mathrm{Spec}_{S}(\mathcal{A})\rightarrow S$
equipped with a $\mathbb{G}_{a,S}$-action is endowed with an increasing
exhaustive filtration by its $\mathcal{O}_{S}$-submodules 
\[
\mathcal{F}_{n}=(\mu^{*})^{-1}(\mathcal{A}\otimes_{\mathcal{O}_{S}}\mathcal{O}_{S}[t]_{\leq n})=\bigcap_{i>n}\mathcal{K}erD^{(i)},\quad n\in\mathbb{Z}_{\geq0},
\]
consisting of elements whose image by $\mu^{*}$ are polynomials with
coefficients in $\mathcal{A}$ of degree less than or equal to $n$
in the variable $t$. The following lemma, whose proof is left to
the reader, records some basic properties of this filtration.
\begin{lem}
\label{lem:Filtration-Properties}With the above notation, the following
hold:

a) For every $m,n\geq0$, we have $\mathcal{F}_{m}\cdot\mathcal{F}_{n}\subseteq\mathcal{F}_{m+n}$,
where $\cdot$ denotes the product law in $\mathcal{A}$,

b) For every $i\geq0$ and $n\geq0$, we have $D^{(i)}\mathcal{F}_{n}\subseteq\mathcal{F}_{n-i}$,

c) The submodule $\mathcal{F}_{0}$ is an $\mathcal{O}_{S}$-subalgebra
of $\mathcal{A}$ which coincides with the $\mathcal{O}_{S}$-algebra
$\mathcal{A}_{0}=\mathcal{A}^{\mathbb{G}_{a,S}}$ of germs of $\mathbb{G}_{a,S}$-invariant
morphisms $X\rightarrow\mathbb{A}_{S}^{1}$,

d) Each $\mathcal{F}_{n}$ is naturally endowed with an additional
structure of $\mathcal{A}_{0}$-module,

e) For every $n\geq0$, the image of $\mu^{*}|_{\mathcal{F}_{n}}:\mathcal{F}_{n}\rightarrow\mathcal{A}\otimes\mathcal{O}_{S}[t]_{\leq n}$
is contained in $\mathcal{F}_{n}\otimes_{\mathcal{O}_{S}}\mathcal{O}_{S}[t]_{\leq n}$.
Moreover $\mu^{*}|_{\mathcal{A}_{0}}:\mathcal{A}_{0}\rightarrow\mathcal{A}_{0}\otimes_{\mathcal{O}_{S}}\mathcal{O}_{S}[t]_{\leq0}\cong\mathcal{A}_{0}$
is an isomorphism onto its image.
\end{lem}

\subsection{\label{subsec:Rees-Graded-mor}Rees algebras and associated graded
homomorphisms}
\begin{defn}
Let $f:X=\mathrm{Spec}_{S}(\mathcal{A})\rightarrow S$ be an affine
$S$-scheme endowed with a $\mathbb{G}_{a,S}$-action $\mu:\mathbb{G}_{a,S}\times_{S}X\rightarrow X$,
let $D=\{D^{(i)}\}_{i\geq0}$ and let $\left\{ \mathcal{F}_{n}=\bigcap{}_{i>n}\mathcal{K}erD^{(i)}\right\} _{n\geq0}$
be the corresponding $\mathcal{O}_{S}$-LFIHD and filtration of $\mathcal{A}$
respectively.

1) The \emph{Rees algebra} of $(X,\mu)$ is the sheaf of graded $\mathcal{A}_{0}$-algebras
$\mathcal{R}(X,\mu)=\bigoplus_{n=0}^{\infty}\mathcal{F}_{n}$, equipped
with the multiplication induced by that of $\mathcal{A}$.

2) The\emph{ associated graded algebra} of $(X,\mu)$ is the sheaf
of graded $\mathcal{A}_{0}$-algebras $\mathcal{G}r(X,\mu)=\bigoplus_{n=0}^{\infty}\mathcal{F}_{n}/\mathcal{F}_{n-1}$,
where by convention $\mathcal{F}_{-1}=\{0\}$, equipped with the multiplication
induced by that of $\mathcal{A}$. 
\end{defn}

The collections of inclusions $\gamma_{n}:\mathcal{A}_{0}=\mathcal{F}_{0}\hookrightarrow\mathcal{F}_{n}$
and $\eta_{n}:\mathcal{F}_{n}\hookrightarrow\mathcal{A}$, $n\geq0$
induce respective injective graded $\mathcal{A}_{0}$-algebra homomorphisms
\begin{equation}
\mathcal{A}_{0}[\upsilon]=\bigoplus_{n=0}^{\infty}\mathcal{A}_{0}\stackrel{\gamma}{\longrightarrow}\mathcal{R}(X,\mu)\stackrel{\eta}{\longrightarrow}\bigoplus_{n=0}^{\infty}\mathcal{A}\stackrel{\cong}{\rightarrow}\mathcal{A}[\upsilon]\label{eq:Graded-inclusions}
\end{equation}
of degree $0$, where $\mathcal{A}[\upsilon]\rightarrow\bigoplus_{n=0}^{\infty}\mathcal{A}$
is the isomorphism of graded $\mathcal{A}$-algebras which maps the
variable $\upsilon$ to the constant section $1\in\Gamma(S,\mathcal{O}_{S})\subset\Gamma(S,\mathcal{A}_{0})\subset\Gamma(S,\mathcal{A})$
viewed in degree $1$. This provides an identification of $\mathcal{R}(X,\mu)$
with the $\mathcal{A}_{0}[\upsilon]$-subalgebra 
\[
\bigoplus_{n=0}^{\infty}\mathcal{F}_{n}\upsilon^{n}\subset\mathcal{A}[\upsilon]
\]
consisting of polynomials $p=\sum a_{n}\upsilon^{n}$ in the variable
$\upsilon$ such that $a_{n}\in\mathcal{F}_{n}\subset\mathcal{A}$
for every $n\geq0$.
\begin{lem}
With the above notation, the following hold:

1) The kernel of the surjective graded $\mathcal{A}_{0}$-algebra
homomorphism $q_{0}:\mathcal{R}(X,\mu)\rightarrow\mathcal{G}r(X,\mu)$
of degree $0$ induced by the collection of quotient homomorphisms
$\mathcal{F}_{n}\rightarrow\mathcal{F}_{n}/\mathcal{F}_{n-1}$, $n\geq0$,
is equal to the homogeneous ideal sheaf $\upsilon\mathcal{R}(X,\mu)$
of $\mathcal{R}(X,\mu)$. 

2) The quotient of $\mathcal{R}(X,\mu)$ by the ideal sheaf $(1-\upsilon)\mathcal{R}(X,\mu)$
is canonically isomorphic to $\mathcal{A}$, and the restriction of
quotient homomorphism $q_{1}:\mathcal{R}(X,\mu)\rightarrow\mathcal{A}$
to each homogeneous piece $\mathcal{F}_{n}$ is an isomorphism onto
its image in $\mathcal{A}$.
\end{lem}

\begin{proof}
By definition, we have 
\[
\mathcal{R}(X,\mu)/\upsilon\mathcal{R}(X,\mu)\cong\bigoplus_{n\geq0}(\mathcal{F}_{n}\upsilon^{n})/(\mathcal{F}_{n-1}\upsilon^{n})\cong\bigoplus_{n\geq0}(\mathcal{F}_{n}/\mathcal{F}_{n-1})\upsilon^{n}\cong\mathcal{G}r(X,\mu).
\]
The injective homomorphism $\eta:\mathcal{R}(X,\mu)\rightarrow\mathcal{A}[\upsilon]$
induces an injective homomorphism 
\[
\eta_{(\upsilon)}:\mathcal{R}(X,\mu)_{(\upsilon)}\rightarrow\mathcal{A}[\upsilon]_{(\upsilon)}
\]
between the degree $0$ parts of the localizations of $\mathcal{R}(X,\mu)$
and $\mathcal{A}[\upsilon]$ respectively with respect to the homogeneous
element $\upsilon\in\mathcal{F}_{1}$. By \cite[Proposition 2.2.5]{EGAII},
we have canonical isomorphisms 
\[
\mathcal{R}(X,\mu)/(1-\upsilon)\mathcal{R}(X,\mu)\rightarrow\mathcal{R}(X,\mu)_{(\upsilon)},\quad\overline{\mathcal{F}_{n}}\ni\overline{f}_{n}\mapsto f_{n}/\upsilon^{n}
\]
and 
\[
\mathcal{A}[\upsilon]/(1-\upsilon)\mathcal{A}[\upsilon]\rightarrow\mathcal{A}[\upsilon]_{(\upsilon)}\cong\mathcal{A},\quad\overline{a_{n}\upsilon^{n}}\mapsto a_{n}.
\]
Via these canonical isomorphisms, the homomorphism $q_{1}:\mathcal{R}(X,\mu)\rightarrow\mathcal{A}$
coincides with the composition of the localization homomorphism 
\[
\mathcal{R}(X,\mu)\rightarrow\mathcal{R}(X,\mu)_{(\upsilon)},\quad\mathcal{F}_{n}\ni f_{n}\mapsto f_{n}/\upsilon^{n}
\]
with $\eta_{(\upsilon)}:\mathcal{R}(X,\mu)_{(\upsilon)}\rightarrow\mathcal{A}[\upsilon]_{(\upsilon)}$.
The second assertion then follows since $\mathcal{A}=\bigcup_{n\geq0}\mathcal{F}_{n}$
by hypothesis. 
\end{proof}
Let $f_{0}:X_{0}=\mathrm{Spec}_{S}(\mathcal{A}_{0})\rightarrow S$
be the relative spectrum of the $\mathcal{O}_{S}$-subalgebra $\mathcal{A}_{0}$
of $\mathcal{A}$. The closed immersions $i_{1}:X\hookrightarrow\mathrm{Spec}_{S}(\mathcal{R}(X,\mu))$
and $i_{0}:\mathrm{Spec}_{S}(\mathcal{G}r(X,\mu))\hookrightarrow\mathrm{Spec}_{S}(\mathcal{R}(X,\mu))$
defined by $q_{1}$ and $q_{0}$ respectively fit in a commutative
diagram

\[\xymatrix{\mathrm{Spec}_{S}(\mathcal{G}r(X,\mu)) \ar[r]^{i_{0}} \ar[d] & \mathrm{Spec}_{S}(\mathcal{R}(X,\mu)) \ar[d] &  X \ar[d] \ar[l]_{i_1} \\    X_0\cong \mathrm{Spec}_{S}(\mathcal{A}_{0}[\upsilon]/(\upsilon))  \ar[r] & \mathrm{Spec}_{S}(\mathcal{A}_{0}[\upsilon]) &  \mathrm{Spec}_{S}(\mathcal{A}_{0}[\upsilon]/(1-\upsilon)) \cong X_0 \ar[l]}\]
whose left-hand and right-hand squares are cartesian.
\begin{lem}
\label{lem:Graded-Hom-Open-Immersion} The graded homomorphisms $q_{0}:\mathcal{R}(X,\mu)\rightarrow\mathcal{G}r(X,\mu)$
and $\eta:\mathcal{R}(X,\mu)\rightarrow\mathcal{A}[\upsilon]$ induce
respectively: 

1) A closed immersion $\overline{i}_{0}:\mathrm{Proj}_{S}(\mathcal{G}r(X,\mu))\rightarrow\mathrm{Proj}_{S}(\mathcal{R}(X,\mu))$
with image equal to Weil divisor $V_{+}(\upsilon)$,

2) An open embedding $j:X\cong\mathrm{Proj}_{S}(\mathcal{A}[\upsilon])\hookrightarrow\mathrm{Proj}_{S}(\mathcal{R}(X,\mu))$
with image equal to the open subset $D_{+}(\upsilon)=\mathrm{Proj}_{S}(\mathcal{R}(X,\mu))\setminus V_{+}(\upsilon)$. 
\end{lem}

\begin{proof}
The first assertion is clear. For the second, we observe that 
\[
\mathrm{Proj}(\eta):\mathrm{Proj}_{S}(\mathcal{A}[\upsilon])\rightarrow\mathrm{Proj}_{S}(\mathcal{R}(X,\mu))
\]
is the composition of the canonical isomorphisms 
\[
\mathrm{Proj}_{S}(\mathcal{A}[\upsilon])\cong\mathrm{Spec}_{S}(\mathcal{A})\cong\mathrm{Spec}_{S}(\mathcal{A}[\upsilon]_{(\upsilon)})\stackrel{\mathrm{Spec(\eta_{(\upsilon)})}}{\longrightarrow}\mathrm{Spec}_{S}(\mathcal{R}(X,\mu)_{(\upsilon)})=D_{+}(\upsilon)
\]
with the embedding $D_{+}(\upsilon)\hookrightarrow\mathrm{Proj}_{S}(\mathcal{R}(X,\mu))$.
\end{proof}

\subsection{\label{subsec:Canonical-Ga-actions}Associated canonical additive
group actions}

Let $f:X=\mathrm{Spec}_{S}(\mathcal{A})\rightarrow S$ be an affine
$S$-scheme endowed with a $\mathbb{G}_{a,S}$-action $\mu:\mathbb{G}_{a,S}\times_{S}X\rightarrow X$.
Let $D=\left\{ D^{(i)}\right\} _{i\geq0}$ be the corresponding $\mathcal{O}_{S}$-LFIHD
of $\mathcal{A}$ and $\mathcal{R}(X,\mu)=\bigoplus_{n=0}^{\infty}\mathcal{F}_{n}$
be the Rees algebra of $(X,\mu)$.

Since $D^{(i)}\mathcal{F}_{n}\subseteq\mathcal{F}_{n-i}$ by Lemma
(\ref{lem:Filtration-Properties}) b), the $\mathcal{O}_{S}$-LFIHD
$D$ induces an homogeneous $\mathcal{O}_{S}$-LFIHD 
\[
\mathcal{R}(D):\mathcal{R}(X,\mu)\rightarrow\mathcal{R}(X,\mu)
\]
of degree $-1$ with respect to the grading of $\mathcal{R}(X,\mu)$,
defined by $\mathcal{R}(D)^{(i)}|_{\mathcal{F}_{n}}=D^{(i)}|_{\mathcal{F}_{n}}$
for every $i,n\geq0$. Furthermore, $\mathcal{R}(D)$ induces via
the quotient morphism $q_{0}:\mathcal{R}(X,\mu)\rightarrow\mathcal{G}r(X,\mu)$
an $\mathcal{O}_{S}$-LFIHD $\mathrm{gr}(D)$ of $\mathcal{G}r(X,\mu)$
which is also homogeneous of degree $-1$. By construction, we have
the following: 
\begin{lem}
Let $f:X=\mathrm{Spec}_{S}(\mathcal{A})\rightarrow S$ be an affine
$S$-scheme endowed with a $\mathbb{G}_{a,S}$-action $\mu:\mathbb{G}_{a,S}\times_{S}X\rightarrow X$
with associated $\mathcal{O}_{S}$-LFIHD $D$ of $\mathcal{A}$. Then
the closed immersions 
\[
X\stackrel{i_{1}}{\hookrightarrow}\mathrm{Spec}_{S}(\mathcal{R}(X,\mu))\stackrel{i_{0}}{\hookleftarrow}\mathrm{Spec}_{S}(\mathcal{G}r(X,\mu))
\]
are equivariant for the $\mathbb{G}_{a,S}$-actions on $X$, $\mathrm{Spec}_{S}(\mathcal{R}(X,\mu))$
and $\mathrm{Spec}_{S}(\mathcal{G}r(X,\mu))$ associated respectively
to the $\mathcal{O}_{S}$-LFIHD $D$, $\mathcal{R}(D)$ and $\mathrm{gr}(D)$.
\end{lem}

For every $i\geq0$, let $\upsilon^{i}\mathcal{R}(D)^{(i)}$ denote
the $\mathcal{O}_{S}$-linear differential operator of $\mathcal{R}(X,\mu)$
whose restriction to $\mathcal{F}_{n}$ is equal to the composition
of $\mathcal{R}(D)^{(i)}:\mathcal{F}_{n}\rightarrow\mathcal{F}_{n-i}$
with the natural inclusion $\mathcal{F}_{n-i}\hookrightarrow\mathcal{F}_{n}$.
The collection $\upsilon\mathcal{R}(D)=\left\{ \upsilon^{i}\mathcal{R}(D)^{(i)}\right\} _{i\geq0}$
is then an $\mathcal{O}_{S}$-LFIHD of $\mathcal{R}(X,\mu)$ which
is homogeneous of degree $0$ with respect to the grading. Note that
it induces the trivial $\mathcal{O}_{S}$-LFIHD on $\mathcal{G}r(X,\mu)$.
Since $\upsilon\mathcal{R}(D)$ is homogeneous of degree $0$, it
defines a $\mathbb{G}_{a,S}$-action on $\mathrm{Spec}_{S}(\mathcal{R}(X,\mu))$
which commutes with the $\mathbb{G}_{m,S}$-action associated to the
grading of $\mathcal{R}(X,\mu)$. The $\mathbb{G}_{a,S}$-action on
$\mathrm{Spec}_{S}(\mathcal{R}(X,\mu))$ thus descends to a $\mathbb{G}_{a,S}$-action
on $\mathrm{Proj}_{S}(\mathcal{R}(X,\mu))$.
\begin{lem}
\label{lem:Projective-Ga-action-extension}Let $f:X=\mathrm{Spec}_{S}(\mathcal{A})\rightarrow S$
be an affine $S$-scheme endowed with a $\mathbb{G}_{a,S}$-action
$\mu:\mathbb{G}_{a,S}\times_{S}X\rightarrow X$ with associated $\mathcal{O}_{S}$-LFIHD
$D$ of $\mathcal{A}$. Then the open embedding $j:X\hookrightarrow\mathrm{Proj}_{S}(\mathcal{R}(X,\mu))$
of Lemma \ref{lem:Graded-Hom-Open-Immersion} is equivariant for the
$\mathbb{G}_{a,S}$-actions on $X$ and $\mathrm{Proj}_{S}(\mathcal{R}(X,\mu))$
determined respectively by $D$ and the homogeneous $\mathcal{O}_{S}$-LFIHD
$\upsilon\mathcal{R}(D)$.
\end{lem}

\begin{proof}
Viewing $\mathcal{R}(X,\mu)$ as a subalgebra of $\mathcal{A}[\upsilon]$
via the injective homomorphism $\eta:\mathcal{R}(X,\mu)\rightarrow\mathcal{A}[\upsilon]$
in (\ref{eq:Graded-inclusions}), the $\mathcal{O}_{S}$-LFIHD $\upsilon\mathcal{R}(D)$
coincides with restriction to $\mathcal{R}(X,\mu)$ of the $\mathcal{O}_{S}$-LFIHD
$D\otimes\mathrm{id}$ of $\mathcal{A}[\upsilon]=\mathcal{A}\otimes_{\mathcal{O}_{S}}\mathcal{O}_{S}[\upsilon]$
corresponding to the $\mathbb{G}_{a,S}$-action on $X\times_{S}\mathbb{A}_{S}^{1}=\mathrm{Spec}_{S}(\mathcal{A}[\upsilon])$
defined as the product of the $\mathbb{G}_{a,S}$-action $\mu$ on
$X$ with the trivial $\mathbb{G}_{a,S}$-action on the second factor.
The open embedding $j=\mathrm{Proj}(\eta):\mathrm{Proj}_{S}(\mathcal{A}[\upsilon])\rightarrow\mathrm{Proj}_{S}(\mathcal{R}(X,\mu))$
of Lemma \ref{lem:Graded-Hom-Open-Immersion} is thus equivariant
for the corresponding $\mathbb{G}_{a,S}$-actions. The assertion follows
since the canonical isomorphism $X\cong\mathrm{Proj}_{S}(\mathcal{A}[\upsilon])$
is equivariant for the $\mathbb{G}_{a,S}$-actions determined by $D$
and $D\otimes\mathrm{id}$ respectively.
\end{proof}

\subsection{\label{subsec:Behavior-EquivMor}Behavior with respect to equivariant
morphisms}

Let $h:S'\rightarrow S$ be a morphism of schemes or algebraic spaces
and let $\tilde{h}:\mathbb{G}_{a,S'}\rightarrow\mathbb{G}_{a,S}$
be the homomorphism of group schemes it induces. Let $f:X\rightarrow S$
(resp. $f':X'\rightarrow S'$) be an affine $S$-scheme (resp. affine
$S'$-scheme) and assume that $X$ and $X'$ are endowed with actions
$\mu$ and $\mu'$ of $\mathbb{G}_{a,S}$ and $\mathbb{G}_{a,S'}$
respectively.
\begin{defn}
\label{def:Equivariant-Morphism}With the above notation, a morphism
$g:X'\rightarrow X$ such that $h\circ f'=f\circ g$ is called $\mathbb{G}_{a}$-\emph{equivariant}
if the following diagram commutes \begin{equation}
\label{eq:Equivariant-Mor-diagram}
\xymatrix{\mathbb{G}_{a,S'}\times_{S'}X' \ar[r]^-{\mu'} \ar[d]_{\tilde{h}\times g}  & X' \ar[d]^{g} \ar[r]^{f'} & S' \ar[d]^{h}\\ \mathbb{G}_{a,S}\times_{S}X \ar[r]^-{\mu} & X \ar[r]^{f} & S.}
\end{equation}
\end{defn}

Letting $\mathcal{A}=f_{*}\mathcal{O}_{X}$ and $\mathcal{A}'=f'_{*}\mathcal{O}_{X'}$,
a morphism $g:X'\rightarrow X$ such that $h\circ f'=f\circ g$ is
uniquely determined by its $\mathcal{O}_{S}$-algebra co-morphism
$g^{*}:\mathcal{A}\rightarrow h_{*}\mathcal{A}'$. Let $D=\left\{ D^{(i)}\right\} _{i\geq0}$
and $D'=\left\{ {D'}^{(i)}\right\} _{i\geq0}$ be the $\mathcal{O}_{S}$-LFIHD
and $\mathcal{O}_{S'}$-LFIHD determining the actions $\mu$ and $\mu'$,
and let $\left\{ \mathcal{F}_{n}\right\} _{n\geq0}$ and $\left\{ \mathcal{F}'_{n}\right\} _{\geq0}$
be the associated ascending filtrations of $\mathcal{A}$ and $\mathcal{A}'$
respectively.
\begin{lem}
\label{lem:EquivariantMor-Charac}A morphism $g:X'\rightarrow X$
such that $h\circ f'=f\circ g$ is $\mathbb{G}_{a}$-equivariant if
and only if it satisfies the following equivalent conditions:

1) For every $i\geq0$, $h_{*}{D'}^{(i)}\circ g^{*}=g^{*}\circ D^{(i)}$,

2) For every $n\geq0$, $g^{*}\mathcal{F}_{n}\subseteq h_{*}\mathcal{F}'_{n}$.
\end{lem}

\begin{proof}
By definition, the morphism $\tilde{h}:\mathbb{G}_{a,S'}\rightarrow\mathbb{G}_{a,S}$
is determined by the homomorphism 
\[
h^{*}\otimes\mathrm{id}:\mathcal{O}_{S}[t]=\mathcal{O}_{S}\otimes_{\mathbb{Z}}\mathbb{Z}[t]\rightarrow(h_{*}\mathcal{O}_{S'})\otimes_{\mathbb{Z}}\mathbb{Z}[t]=\tilde{h}_{*}(\mathcal{O}_{S'}[t]),
\]
where $h^{*}:\mathcal{O}_{S}\rightarrow h_{*}\mathcal{O}_{S'}$ is
the $\mathcal{O}_{S}$-module homomorphism in the definition of $h$.
The commutativity of the diagram \eqref{eq:Equivariant-Mor-diagram}
is then equivalent to that of the following diagram of $\mathcal{O}_{S}$-algebra
homomorphisms \[\xymatrix{\mathcal{A} \ar[d]^{g^*} \ar[r]^{\mu^*} & \mathcal{A}\otimes_{\mathcal{O}_{S}}\mathcal{O}_{S}[t] \ar[d]_{g^{*}\otimes\tilde{h}^{*}} \\ h_{*}\mathcal{A}' \ar[r]^-{h_{*}{\mu'}^{*}} & h_{*}(\mathcal{A}\otimes_{\mathcal{O}_{S'}}\mathcal{O}_{S'}[t])=h_{*}\mathcal{A}\otimes_{h_{*}\mathcal{O}_{S'}}\tilde{h}_{*}\mathcal{O}_{S'}[t]  }\]from
which the claimed equivalences follow.
\end{proof}
By Lemma \ref{lem:EquivariantMor-Charac}, the comorphism $g^{*}:\mathcal{A}\rightarrow h_{*}\mathcal{A}'$
of a $\mathbb{G}_{a}$-equivariant morphism $g:X'\rightarrow X$ is
thus a homomorphism of filtered $\mathcal{O}_{S}$-algebras of degree
$0$ with respect to the filtrations $\left\{ \mathcal{F}_{n}\right\} _{n\geq0}$
and $\left\{ h_{*}\mathcal{F}_{n}'\right\} _{n\geq0}$ associated
to the actions $\mu$ and $\mu'$ respectively. As a consequence,
$g^{*}$ induces a homomorphism of $\mathcal{O}_{S}$-algebras $g_{0}^{*}:\mathcal{A}_{0}\rightarrow h_{*}\mathcal{A}_{0}'$
and homomorphisms of graded algebras 
\[
\begin{array}{ccc}
\mathcal{R}(g):\mathcal{R}(X,\mu)\rightarrow h_{*}\mathcal{R}(X',\mu') & \textrm{and} & \mathrm{gr}(g):\mathcal{G}r(X,\mu)\rightarrow h_{*}\mathcal{G}r(X',\mu'),\end{array}
\]
both of degree $0$. Furthermore, with the notation of (\ref{eq:Graded-inclusions}),
we have a commutative diagram \[\xymatrix{\mathcal{A}_{0}[\upsilon] \ar[r]^{\gamma} \ar[d]_{\mathrm{Sym}^{\cdot}g_{0}^{*}} & \mathcal{R}(X,\mu) \ar[r]^{\eta} \ar[d]^{ \mathcal{R}(g)} & \mathcal{A}[\upsilon] \ar[d]^{\mathrm{Sym}^{\cdot}g^{*}} \\ h_{*}\mathcal{A}_{0}[\upsilon] \ar[r]^-{h_*\gamma} & h_{*}\mathcal{R}(X',\mu') \ar[r]^{h_*\eta'} & h_{*}\mathcal{A}'[\upsilon].}\] 
The following proposition is a direct consequence of the definitions
given in subsection \ref{subsec:Canonical-Ga-actions}.
\begin{prop}
\label{prop:EquivariantMor-Rees}Let $f:X=\mathrm{Spec}_{S}(\mathcal{A})\rightarrow S$
and $f':X'=\mathrm{Spec}_{S'}(\mathcal{A}')\rightarrow S'$ be affine
schemes over $S$ and $S'$, endowed respectively with a $\mathbb{G}_{a,S}$-action
$\mu:\mathbb{G}_{a,S}\times_{S}X\rightarrow X$ and a $\mathbb{G}_{a,S'}$-action
$\mu':\mathbb{G}_{a,S'}\times_{S'}X'\rightarrow X'$. Let $D$ and
$D'$ be the associated $\mathcal{O}_{S}$-LFIHD and $\mathcal{O}_{S'}$-LFIHD.
Let $h:S'\rightarrow S$ be a morphism and let $g:X'\rightarrow X$
be a $\mathbb{G}_{a}$-equivariant morphism. Then the following hold:

1) The diagram \[\xymatrix{\mathrm{Spec}_{S'}(\mathcal{G}r(X',\mu')) \ar[r]^{i'_{0}} \ar[d]_{\mathrm{gr}(g)} & \mathrm{Spec}_{S'}(\mathcal{R}(X',\mu')) \ar[d]_{\mathcal{R}(g)} &  X'=\mathrm{Spec}_S(\mathcal{A}') \ar[d]^{g} \ar[l]_-{i_1'} \\ \mathrm{Spec}_{S}(\mathcal{G}r(X,\mu)) \ar[r]^{i_{0}} & \mathrm{Spec}_{S}(\mathcal{R}(X,\mu))  &  X=\mathrm{Spec}_S(\mathcal{A})  \ar[l]_-{i_1}  }\]is
commutative and equivariant for the $\mathbb{G}_{a}$-actions defined
by the $\mathcal{O}_{S'}$-LFIHD $\mathrm{g}r(D')$, $D'$ and $\mathcal{R}(D')$
and the $\mathcal{O}_{S}$-LFIHD $\mathrm{gr}(D)$, $D$ and $\mathcal{R}(D)$
respectively.

2) The diagram \[\xymatrix{X' \ar[d]_{g} \ar[r]^-{j'} & \mathrm{Proj}_{S'}(\mathcal{R}(X',\mu')) \ar[d]^{\mathrm{Proj}(\mathcal{R}(g))} \ar[r] & \mathrm{Spec}_{S'}(\mathcal{A}_{0}') \ar[d]^{g_0=\mathrm{Spec}(g_{0}^{*})} \\ X \ar[r]^-{j} & \mathrm{Proj}_{S}(\mathcal{R}(X,\mu)) \ar[r] &  \mathrm{Spec}_{S}(\mathcal{A}_{0})}\]  is
commutative and equivariant for the $\mathbb{G}_{a}$-actions defined
by the $\mathcal{O}_{S'}$-LFIHD $D'$ and $\upsilon\mathcal{R}(D')$
and the $\mathcal{O}_{S}$-LFIHD $D$ and $\upsilon\mathcal{R}(D)$
respectively.
\end{prop}

\subsection{Rees algebras of $\mathbb{G}_{a}$-torsors}

Recall that a $\mathbb{G}_{a,S}$-torsor is an $S$-scheme $f:P\rightarrow S$
endowed with a $\mathbb{G}_{a,S}$-action $\mu:\mathbb{G}_{a,S}\times_{S}P\rightarrow P$
which, \'etale locally over $S$, is equivariantly isomorphic to
$\mathbb{G}_{a,S}$ acting on itself by translations. In particular,
$P$ is an affine $S$-scheme of finite type. Let $\mathcal{A}=f_{*}\mathcal{O}_{P}$,
and let $D$ and $\left\{ \mathcal{F}_{n}\right\} _{n\geq0}$ be the
$\mathcal{O}_{S}$-LFIHD and the ascending filtration of $\mathcal{A}$
associated to the $\mathbb{G}_{a,S}$-action $\mu$. Since $P$ is
\'etale locally isomorphic to $\mathbb{G}_{a,S}$ acting on itself
by translations, we have $\mathcal{F}_{0}=\mathcal{A}_{0}=\mathcal{A}^{\mathbb{G}_{a,S}}=\mathcal{O}_{S}$. 
\begin{prop}
\label{prop:Ga-torsors-Rees}With the above notation, the following
hold:

a) The $\mathcal{O}_{S}$-module $\mathcal{F}_{1}$ is an \'etale
locally free sheaf of rank $2$ and we have an exact sequence of $\mathcal{O}_{S}$-modules
\begin{equation}
0\rightarrow\mathcal{O}_{S}=\mathcal{F}_{0}\hookrightarrow\mathcal{F}_{1}\stackrel{D^{(1)}}{\rightarrow}\mathcal{F}_{0}\rightarrow0.\label{eq:Torso-degree1-Extension}
\end{equation}

b) The Rees algebra $\mathcal{R}(P,\mu)$ is canonically isomorphic
to the symmetric algebra $\mathrm{Sym}^{\cdot}\mathcal{F}_{1}=\bigoplus_{n=0}^{\infty}\mathrm{Sym}^{n}\mathcal{F}_{1}$
of $\mathcal{F}_{1}$.

c) The open immersion $j:P\hookrightarrow\mathrm{Proj}_{S}(\mathcal{R}(P,\mu))$
coincides with the open immersion of $P$ in the projective bundle
$p:\mathbb{P}(\mathcal{F}_{1})=\mathrm{Proj}_{S}(\mathrm{Sym}^{\cdot}\mathcal{F}_{1})\rightarrow S$
as the complement of the section $S\rightarrow\mathbb{P}(\mathcal{F}_{1})$
determined by the surjective homomorphism $D^{(1)}:\mathcal{F}_{1}\rightarrow\mathcal{O}_{S}$.
\end{prop}

\begin{proof}
Since the surjectivity of the homomorphisms $D^{(1)}:\mathcal{F}_{1}\rightarrow\mathcal{F}_{0}$
and $\mathrm{Sym}^{\cdot}\mathcal{F}_{1}\rightarrow\mathcal{R}(P,\mu)$
are local properties on $S$ with respect to the \'etale topology,
to prove a) and b), it suffices to consider the case where $X\rightarrow S$
is the trivial $\mathbb{G}_{a,S}$-torsor $\mathrm{Spec}_{S}(\mathcal{O}_{S}[t])$
with the $\mathbb{G}_{a,S}$-action given by the group structure $m:\mathbb{G}_{a,S}\times_{S}\mathbb{G}_{a,S}\rightarrow\mathbb{G}_{a,S}$.
Here the corresponding $\mathcal{O}_{S}$-LFIHD $D$ is given by the
collection of differential operators
\[
D^{(i)}=\frac{1}{i!}\frac{\partial^{i}}{\partial t^{i}}|_{t=0},\quad i\geq0
\]
which associate to a polynomial $p(t)$ the $i$-th term of its Taylor
expansion at $0$. We thus have $\mathcal{F}_{n}=\mathcal{O}_{S}[t]_{\leq n}$,
$n\geq0$. In particular, $\mathcal{F}_{1}$ is the free $\mathcal{O}_{S}$-module
of rank $2$ generated by $1$ and $t$, with $D^{(1)}(t)=1$, which
proves a). We then have $\mathcal{F}_{n}\cong\mathrm{Sym}^{n}\mathcal{F}_{1}$
as $\mathcal{O}_{S}$-modules, and so $\mathcal{R}(\mathbb{G}_{a,S},m)=\mathrm{Sym}^{\cdot}\mathcal{F}_{1}\cong\mathcal{O}_{S}[\upsilon,t]$
from which assertion b) follows. Note that the $\mathcal{O}_{S}$-LFIHD
$\mathcal{R}(D)$ and $\upsilon\mathcal{R}(D)$ on $\mathrm{Sym}^{\cdot}\mathcal{F}_{1}$
are then given locally by $\mathcal{R}(D)^{(i)}=\frac{1}{i!}\frac{\partial^{i}}{\partial t^{i}}|_{t=0}$
and $(\upsilon\mathcal{R}(D))^{(i)}=\upsilon^{i}\frac{1}{i!}\frac{\partial^{i}}{\partial t^{i}}|_{t=0}$
respectively.

The section of $\mathbb{P}(\mathcal{F}_{1})$ defined by the surjective
homomorphism of $\mathcal{O}_{S}$-modules $D^{(1)}:\mathcal{F}_{1}\rightarrow\mathcal{O}_{S}$
is given by the closed immersion 
\[
S\cong\mathrm{Proj}_{S}(\mathcal{O}_{S}[t])\rightarrow\mathbb{P}(\mathcal{F}_{1})=\mathrm{Proj}_{S}(\mathrm{Sym}^{\cdot}\mathcal{F}_{1})
\]
determined by the surjective homomorphism of graded $\mathcal{O}_{S}$-algebras
$\mathrm{Sym}^{.}(D^{(1)}):\mathrm{Sym}^{\cdot}\mathcal{F}_{1}\rightarrow\mathrm{Sym}^{.}\mathcal{O}_{S}\cong\mathcal{O}_{S}[t]$.
By the previous description, the homomorphism $\mathrm{Sym}^{.}(D^{(1)})$
coincides locally over $S$ with the homomorphism $\mathcal{O}_{S}[\upsilon,t]\rightarrow\mathcal{O}_{S}[t]$
with kernel $\upsilon\mathcal{O}_{S}[\upsilon,t]$. It follows that
the kernel of $\mathrm{Sym}^{.}(D^{(1)})$ coincides via the isomorphism
$\mathcal{R}(P,\mu)\cong\mathrm{Sym}^{\cdot}\mathcal{F}_{1}$ with
the homogeneous ideal sheaf $\upsilon\mathcal{R}(P,\mu)$ of $\mathcal{R}(P,\mu)$,
with quotient $\mathcal{R}(P,\mu)/\upsilon\mathcal{R}(P,\mu)\cong\mathcal{G}r(P,\mu)$.
Assertion c) is then a consequence of Lemma \ref{lem:Graded-Hom-Open-Immersion}. 
\end{proof}
\begin{rem}
It follows from Proposition \ref{prop:Ga-torsors-Rees} a) and c)
that the isomorphism class of a $\mathbb{G}_{a,S}$-torsor $f:P\rightarrow S$
in $H_{\mathrm{\acute{e}t}}^{1}(S,\mathbb{G}_{a,S})\cong H_{\mathrm{\acute{e}t}}^{1}(S,\mathcal{O}_{S})$
coincides via the isomorphism $H_{\mathrm{\acute{e}t}}^{1}(S,\mathcal{O}_{S})\cong\mathrm{Ext}_{\mathrm{\acute{e}t}}^{1}(\mathcal{O}_{S},\mathcal{O}_{S})$
with the class of the dual of the extension (\ref{eq:Torso-degree1-Extension})
in Proposition \ref{prop:Ga-torsors-Rees}.
\end{rem}

\begin{example}
\label{exa:Relative-Rees-SL2} (See also \cite[Proposition 1.2]{DF14}
and \cite[Proposition 1]{He15}). Let $(S',o)$ be a pair consisting
of the spectrum of $2$-dimensional regular local ring and its closed
point $o$, and let $\rho:P\rightarrow S=S'\setminus\{o\}$ be a $\mathbb{G}_{a,S}$-torsor.
Let $\left\{ \mathcal{F}_{n}\right\} _{n\geq0}$ be the ascending
filtration of $\mathcal{A}=\rho_{*}\mathcal{O}_{P}$ associated to
the $\mathbb{G}_{a,S}$-action $\mu_{P}:\mathbb{G}_{a,S}\times_{S}P\rightarrow P$
on $P$. By Proposition \ref{prop:Ga-torsors-Rees}, $\mathcal{F}_{1}$
is a locally free sheaf of rank $2$ on $S$, which is in fact free
by virtue of \cite[Corollary 4.1.1]{Hor64}. The Rees algebra $\mathcal{R}(P,\mu)=\bigoplus_{n\geq0}\mathcal{F}_{n}$
is thus isomorphic to the polynomial ring algebra $\mathcal{O}_{S}[u,v]$
in two variables $u$, $v$ over $\mathcal{O}_{S}$. The surjection
$D^{(1)}:\mathcal{F}_{1}\rightarrow\mathcal{F}_{0}=\mathcal{O}_{S}$
maps $u$ and $v$ to respective elements $x$ and $y$ of $\Gamma(S,\mathcal{O}_{S})=\Gamma(S',\mathcal{O}_{S'})$,
which have the property that $(x,y)\mathcal{O}_{S'}|_{S}=\mathcal{O}_{S}$,
and the image of the open immersion 
\[
j:P\hookrightarrow\mathrm{Proj}_{S}(\mathcal{R}(P,\mu))\cong\mathrm{Proj}_{S}(\mathcal{O}_{S}[u,v])\cong S\times_{\mathbb{Z}}\mathbb{P}_{\mathbb{Z}}^{1}
\]
is equal to the complement of the Cartier divisor $B$ with equation
$xv-yu=0$. Letting $B'$ be the closure of $B$ in $S'\times_{\mathbb{Z}}\mathbb{P}_{\mathbb{Z}}^{1}$
we have the following alternative:

1) Either $B'$ fully contains the fiber of $\mathrm{pr}_{S'}:S'\times_{\mathbb{Z}}\mathbb{P}_{\mathbb{Z}}^{1}\rightarrow S'$
over the closed point $o$ and then $P\cong S'\times_{\mathbb{Z}}\mathbb{P}_{\mathbb{Z}}^{1}\setminus B'$
is a nontrivial $\mathbb{G}_{a,S}$-torsor, isomorphic to the closed
subscheme of $S'\times_{\mathbb{Z}}\mathrm{\mathbb{A}_{\mathbb{Z}}^{2}}$
with equation $xv-yu=1$, 

2) Or $B'$ extends to a section of $\mathrm{pr}_{S'}:S'\times_{\mathbb{Z}}\mathbb{P}_{\mathbb{Z}}^{1}\rightarrow S'$
and then $S'\times_{\mathbb{Z}}\mathbb{P}_{\mathbb{Z}}^{1}\setminus B'\cong S'\times_{\mathbb{Z}}\mathbb{A}_{\mathbb{Z}}^{1}$
and 
\[
P\cong S'\times_{\mathbb{Z}}\mathbb{P}_{\mathbb{Z}}^{1}\setminus(B'\cup\mathrm{pr}_{S'}^{-1}(o))\cong S\times_{\mathbb{Z}}\mathbb{A}_{\mathbb{Z}}^{1}
\]
is the trivial $\mathbb{G}_{a,S}$-torsor.
\end{example}

\section{Rees algebras of affine $\mathbb{G}_{a}$-varieties over a field
of characteristic zero}

This section is devoted to the study of Rees algebras in the ``absolute''
case where the base scheme $S$ is the spectrum of a field $k$, which
we further assume to be of characteristic zero for simplicity. We
establish basic additional properties of Rees algebras in this context,
with a special emphasis on their behavior with respect to equivariant
birational morphisms such as the normalization or equivariant affine
modifications. We also study the problem of finite generation of Rees
algebras from both algebraic and geometric viewpoints. Throughout
this section, we denote the additive group $\mathbb{G}_{a,k}$ simply
by $\mathbb{G}_{a}$.

\subsection{Basic properties of global Rees algebras of affine $\mathbb{G}_{a}$-varieties }

\indent\newline Here $S=\mathrm{Spec}(k)$ is the spectrum of an
algebraically closed field $k$ of characteristic zero and $X=\mathrm{Spec}(A)$
is the spectrum of an integral $k$-algebra of finite type. In this
context, a $k$-LFIHD $D=\{D^{(i)}\}_{i\geq0}$ of $A$ is uniquely
determined by $D^{(i)}=\frac{1}{i!}\partial^{i}$ where $\partial=D^{(1)}:A\rightarrow A$
is a $k$-derivation of $A$ such that $A=\bigcup_{n\geq0}(\bigcap_{i>n}\mathrm{Ker}\partial^{i})$.
Since for every $n\geq0$, $\mathrm{Ker}\partial^{n}\subset\mathrm{Ker}\partial^{i}$
for every $i\geq n$, we have in fact $A=\bigcup_{i\geq0}\mathrm{Ker}\partial^{i}$,
i.e. $\partial$ is a locally nilpotent $k$-derivation of $A$ in
the sense of \cite{FreuBook}. Furthermore, the associated ascending
filtration of $A$ consists simply of the $k$-vector subspaces $F_{n}=\mathrm{Ker}\partial^{n+1}$,
$n\geq0$. The subspaces $F_{n}$, which have the natural additional
structure of modules over the ring $F_{0}=A_{0}=\mathrm{Ker}\partial$
of $\mathbb{G}_{a}$-invariants are called the degree modules associated
to $\partial$ in \cite{Fr16,FreuBook}.

The Rees algebra of an affine $k$-variety $X=\mathrm{Spec}(A)$ with
a $\mathbb{G}_{a}$-action determined by a locally nilpotent $k$-derivation
$\partial$ of $A$ is thus equal to the graded algebra 
\[
R(A,\partial)=\bigoplus_{n\geq0}F_{n}=\bigoplus_{n\geq0}\mathrm{Ker}\partial^{n+1}.
\]
We denote by $\mathrm{gr}_{\partial}A$ the associated graded algebra
$\bigoplus_{n\geq0}F_{n}/F_{n-1}$, where by convention $F_{-1}=\{0\}$.
The locally nilpotent $k$-derivation $\partial$ induces a canonical
homogeneous locally nilpotent $k$-derivation $R(\partial)$ of $R(A,\partial)$
of degree $-1$ given in restriction on each homogeneous component
by 
\[
R(\partial)|_{F_{n}}=\partial:F_{n}=\mathrm{Ker}\partial^{n+1}\rightarrow F_{n-1}=\mathrm{Ker}\partial^{n},\quad n\geq0.
\]
It induces a homogeneous locally nilpotent $k$-derivation $\mathrm{gr}(\partial)$
of $\mathrm{gr}_{\partial}A$ of degree $-1$.

As in subsection \ref{subsec:Rees-Graded-mor}, we can view $R(A,\partial)$
as the graded $A_{0}[\upsilon]$-subalgebra $\bigoplus_{n\geq0}F_{n}\upsilon^{n}$
of $A[\upsilon]$. It follows in particular that $R(A,\partial)$
is an integral $k$-algebra. The locally nilpotent $k[\upsilon]$-derivation
$\upsilon R(\partial)$ of $R(A,\partial)$ then coincides with the
restriction to $R(A,\partial)$ of the homogeneous locally nilpotent
$k[\upsilon]$-derivation $\tilde{\partial}$ of $A[\upsilon]$ of
degree $0$ defined by $\tilde{\partial}(\sum a_{i}\upsilon^{i})=\sum\partial(a_{i})\upsilon^{i}$.
\begin{lem}
\label{lem:upsilon-inversion}Let $(A,\partial)$ be a finitely generated
$k$-algebra endowed with a locally nilpotent $k$-derivation $\partial$
and let $R(A,\partial)$ be its Rees algebra. Then 
\[
R(A,\partial)[\upsilon^{-1}]=R(A,\partial)\otimes_{k[\upsilon]}k[\upsilon^{\pm1}]\cong A\otimes_{k}k[\upsilon^{\pm1}]=A[\upsilon^{\pm1}]
\]
and the induced $k[\upsilon^{\pm1}]$-derivations $\upsilon R(\partial)$
and $\tilde{\partial}$ of $R(A,\partial)[\upsilon^{\pm1}]$ and $A[\upsilon^{\pm1}]$
respectively coincide under this isomorphism.
\end{lem}

\begin{proof}
The inclusion $R(A,\partial)[\upsilon^{-1}]\subseteq A[\upsilon^{\pm1}]$
is clear. Conversely, let $x=\upsilon^{-k}(a_{0}+a_{1}\upsilon+\cdots a_{n}\upsilon^{n})\in A[\upsilon^{\pm1}]$.
Since $A=\bigcup_{m\geq1}\mathrm{Ker}\partial^{m}$, there exists
$m_{0}\geq1$ such that $a_{i}\in\mathrm{Ker}\partial^{m_{0}}$ for
every $i$. It follows that $a_{i}\upsilon^{i}=\upsilon^{i-m_{0}}(a_{i}\upsilon^{m_{0}})\in R(A,\partial)[\upsilon^{-1}]$
and then that 
\begin{align*}
x & =\upsilon^{-k}(\upsilon^{-m_{0}}(a_{0}\upsilon^{m_{0}})+\cdots\upsilon^{n-m_{0}}(a_{n}\upsilon^{m_{0}}))\\
 & =\upsilon^{-k-m_{0}}(a_{0}\upsilon^{m_{0}}+a_{1}\upsilon^{m_{0}+1}+\cdots a_{n}\upsilon^{m_{0}+n})
\end{align*}
where for every $i$, $a_{i}\upsilon^{m_{0}+i}\in(\mathrm{Ker}\partial^{m_{0}+i})\upsilon^{m_{0}+i}$
since $a_{i}\in\mathrm{Ker}\partial^{m_{0}}\subset\mathrm{Ker}\partial^{m_{0}+i}$.
Thus $A[\upsilon^{\pm1}]\subseteq R(A,\partial)[\upsilon^{-1}]$.
The fact that the induced derivations coincide follows by construction.
\end{proof}
\begin{lem}
\label{lem:Rees-preslice}Let $(A,\partial)$ be an integral $k$-algebra
endowed with a locally nilpotent $k$-derivation $\partial$ and let
$\{F_{n}\}_{n\geq0}$ be the associated ascending filtration of $A$.
Then for every $s\in F_{1}\setminus F_{0}$, there exists an isomorphism
of graded algebras 
\[
(R(A,\partial)_{\partial s},R(\partial))\cong(R((F_{0})_{\partial s}[s],\tfrac{\partial}{\partial s}),R(\tfrac{\partial}{\partial s}))
\]
where $\partial s\in F_{0}$ is viewed as homogeneous element of degree
$0$ in $R(A,\partial)$.
\end{lem}

\begin{proof}
Since $\partial s\in F_{0}$, it belongs to $\mathrm{Ker}R(\partial)$.
Thus $R(\partial)$ extends in a canonical way to a homogeneous locally
nilpotent $k$-derivation of $R(A,\partial)_{\partial s}=R(A_{\partial s},\partial)$
which we denote by the same symbol. On the other hand, the $A_{0}$-subalgebra
$A_{0}[s]$ of $A$ generated by $s$ is stable under $\partial$,
and $\partial$ restricts on $A_{0}[s]$ to the nonzero locally nilpotent
$A_{0}$-derivation $\frac{\partial}{\partial s}$. Since $\partial s\in A_{0}$,
$\frac{\partial}{\partial s}$ and $\partial$ extend to well-defined
locally nilpotent $k$-derivations of the localizations $A_{0}[s]_{\partial s}=(A_{0})_{\partial s}[s]$
and $A_{\partial s}$ respectively, which we denote again by the same
symbol. By \cite[Principle 11 (d)]{FreuBook}, the inclusion $(A_{0}[s],\tfrac{\partial}{\partial s})\subset(A,\partial)$
induces an isomorphism $((A_{0})_{\partial s}[s],\tfrac{\partial}{\partial s})\cong(A_{\partial s},\partial)$.
This in turns induces the desired isomorphism $R((A_{0})_{\partial s}[s],\tfrac{\partial}{\partial s})\cong R(A_{\partial s},\partial)$
for which $R(\partial)$ coincides with $R(\tfrac{\partial}{\partial s})$.
\end{proof}
\begin{rem}
In the setting of Lemma \ref{lem:Rees-preslice}, it follows in turn
from Proposition \ref{prop:Ga-torsors-Rees} that $R(A,\partial)_{\partial s}$
is canonically isomorphic to the symmetric algebra of the free $(F_{0})_{\partial s}$-submodule
$F_{1}A_{\partial s}\cong(F_{0})_{\partial s}\cdot s\oplus(F_{0})_{\partial s}$
of rank $2$ of $A_{\partial s}$. This yields an isomorphism of graded
algebras
\[
(R(A,\partial)_{\partial s},R(\partial))\cong((F_{0})_{\partial s}[s,\upsilon],\tfrac{\partial}{\partial s})
\]
where $s$ and $\upsilon$ are homogeneous of degree $1$.
\end{rem}

\subsection{Rees algebras and equivariant birational morphisms}

We now consider the behavior of Rees algebras under certain equivariant
birational morphisms between affine varieties. Let $(A,\partial)$
be an integral $k$-algebra endowed with a non-zero locally nilpotent
$k$-derivation and let $A'\subset\mathrm{Frac}(A)$ be its normalization,
i.e. its integral closure in its field of fraction $\mathrm{Frac}(A)$.
By results of Seidenberg and Vasconcelos (see e.g. \cite[Proposition 1.2.15 and Proposition 1.3.37]{vdEBook}),
there exists a unique extension of $\partial$ to a locally nilpotent
$k$-derivation $\partial'$ of $A'$.
\begin{lem}
\label{lem:Normalization-Commute-Rees}With the above notation, the
Rees algebra $R(A',\partial')$ is the normalization of the Rees algebra
$R(A,\partial)$. Furthermore, the unique extension to $R(A',\partial')$
of the canonical homogeneous locally nilpotent $k$-derivation $R(\partial)$
of $R(A,\partial)$ coincides with the canonical homogeneous locally
nilpotent $k$-derivation $R(\partial')$ of $R(A',\partial')$. 
\end{lem}

\begin{proof}
Let $\{F_{n}\}_{n\geq0}$ and $\{F_{n}'\}_{n\geq0}$ be the ascending
filtrations of $A$ and $A'$ associated to $\partial$ and $\partial'$
respectively. Let $R=\bigoplus_{n\geq0}F_{n}\upsilon^{n}\cong R(A,\partial)$
and $R'=\bigoplus_{n\geq0}F_{n}'\upsilon^{n}\cong R(A',\partial')$.
Since $F_{n}=F_{n}'\cap A$ by construction of $\partial'$, we have
the following commutative diagram of inclusions \[\xymatrix {R \ar[d] \ar[r] & R' \ar[d] \\ A[\upsilon] \ar[r] & A'[\upsilon].}\]
By Lemma \ref{lem:upsilon-inversion}, we have $R[\upsilon^{-1}]=A[\upsilon^{\pm1}]$
and $R'[\upsilon^{-1}]=A'[\upsilon^{\pm1}]$, so that $R$ and $R'$
have the same field of fractions. The normalization of $R$ is thus
contained in that of $R'$, and since on the other hand every homogeneous
element $x'\in F_{n}'\subset A'$ is integral over $A$, $R'$ is
contained in the normalization of $R$. It is thus enough to show
that $R'$ is normal. If $h\in\mathrm{Frac}(R')\cong\mathrm{Frac}(A')(\upsilon)$
is integral over $R'$ then it is also integral over $R'[\upsilon^{-1}]\cong A'[\upsilon^{\pm1}]$,
hence belongs to this algebra, $A'[\upsilon^{\pm1}]$ being normal
as $A'$ is normal. It follows that $h=\upsilon^{-\ell}g$ for some
$g\in A'[\upsilon]$ which is integral over $R'$, and it remains
to prove that $R'$ is integrally closed in $A'[\upsilon]$.

Since the inclusion $R'\hookrightarrow A'[\upsilon]$ is a graded
homomorphism, the integral closure of $R'$ in $A'[\upsilon]$ is
a graded $R'$-subalgebra of $A'[\upsilon]$ \cite[§1.8 Proposition 20]{Bour}.
As a consequence, $R'$ is integrally closed in $A'[\upsilon]$ if
and only if every homogeneous element $g=h\upsilon^{n}\in A'[\upsilon]$
which is integral over $R'$ belongs to $R'$. Let $g^{m}+\sum_{i=0}^{m-1}a_{i}g^{i}=0$
be a homogeneous integral dependence relation with coefficients in
$R'$. Since $g^{i}=h^{i}\upsilon^{ni}$, $a_{i}$ is homogeneous
of degree $(m-i)n$, hence is of the form $b_{(m-i)n}\upsilon^{(m-i)n}$
for some $b_{(m-i)n}\in F'_{(m-i)n}$. This implies that the relation
$h^{m}+\sum_{i=0}^{m-1}b_{(m-i)n}h^{i}=0$ holds in $A'$. If $n=0$,
then $h\in A'$ is integral over $A_{0}'$, hence belongs to $A_{0}'$
since the latter is integrally closed in $A'$ by \cite[Proposition 1.13]{FreuBook}.
If $n\geq1$, then by definition $g\in R'$ if and only if $h\in F_{n}'$.
So suppose that $h\in F_{d}'\setminus F_{d-1}'$ for some $d>n$.
Then $h^{m}\in F_{md}'\setminus F_{md-1}'$ but on the other hand
$\sum_{i=0}^{m-1}b_{(m-i)n}h^{i}\in F'_{dm-1}$ as $F'_{(m-i)n+di}=F'_{mn+(d-n)i}$
is contained $F'_{dm-1}$ for every $i=0,\ldots,m-1$. This is absurd,
so $h\in F_{n}'$ and then $g\in R'$.
\end{proof}
\begin{cor}
Let $A$ be an integral normal $k$-algebra. Then for every nonzero
locally nilpotent $k$-derivation $\partial$ of $A$, the Rees algebra
$R(A,\partial)$ is integral and normal.
\end{cor}

Let $(A,\partial)$ be an integral $k$-algebra endowed with a non-zero
locally nilpotent $k$-derivation, let $I\subset A$ be a $\partial$-invariant
ideal and $f\in I$ be a $\partial$-invariant element, so that $\partial f=0$
by \cite[Corollary 1.23]{FreuBook}. Let $\tilde{\partial}$ be the
locally nilpotent $k[t]$-derivation of $A[t]$ of degree $0$ defined
by $\tilde{\partial}(\sum a_{i}t^{i})=\sum\partial(a_{i})t^{i}$ and
let $\overline{\partial}$ be the locally nilpotent $k$-derivation
of $A[t]/(1-ft)$ that it induces. Since $\partial I\subset I$, $\overline{\partial}$
restricts to a locally nilpotent $k$-derivation of the integral $k$-algebra
\[
A[I/f]=(\bigoplus_{n\geq0}I^{n}t^{n})/(1-ft)\subseteq A[t]/(1-ft)\cong A[f^{-1}]
\]
which we denote by $\partial'$. The natural inclusion $A\hookrightarrow A[I/f]$
induces an isomorphism of $k$-algebras $A[f^{-1}]\cong A[I/f][f^{-1}]$.
Furthermore, $A$ is a $\partial'$-invariant subalgebra of $A[I/f]$
and the restriction of $\partial'$ to $A$ is equal to $\partial$.
Following \cite{KaZa99}, we call the pair $(A[I/f],\partial')$ the
\emph{equivariant affine modification of} $(A,\partial)$ with center
at the $\partial$-invariant ideal $I$ and $\partial$-invariant
divisor $f$. 
\begin{lem}
With the above notation, the pair $(R(A[I/f],\partial'),R(\partial'))$
is isomorphic to the equivariant affine modification $(R(A,\partial)[J/f],R(\partial)')$
of $(R(A,\partial),R(\partial))$ with center at the $R(\partial)$-invariant
homogeneous ideal $J\subset R(A,\partial)$ generated by the elements
of $I$ and with $R(\partial)$-invariant divisor $f$.
\end{lem}

\begin{proof}
Every element $h$ of $J$ is a finite sum $h=\sum h_{i}f_{i}$ where
$h_{i}\in I$ and 
\[
f_{i}=\sum f_{ij}\in\bigoplus_{n\geq0}F_{n}=R(A,\partial).
\]
Since $A=\bigcup_{n\geq0}F_{n}$, each $h_{i}$ is homogeneous of
a certain degree when viewed as an element of $R(A,\partial)$. Since
$h_{i}f_{ij}\in I$ for every $i,j$ and $I$ is $\partial$-stable,
it follows that $J$ is an $R(\partial)$-stable homogeneous ideal
of $R(A,\partial)$. Viewing $f$ as a homogeneous element of degree
$0$ in $R(A[I/f],\partial')$, the image of $R(A[I/f],\partial')$
by the injective homogeneous localization homomorphism 
\[
R(A[I/f],\partial')\rightarrow R(A[I/f],\partial')[f^{-1}]=R(A_{f},\partial')=R(A_{f},\partial)=R(A,\partial)_{f}
\]
is equal to the graded subalgebra of $R(A_{f},\partial)$ whose elements
have the form $f^{-\ell}\sum g_{i}$ where $\sum g_{i}\in I^{\ell}R(A,\partial)=J^{\ell}$.
On the other hand, it follows from the definition of $R(A,\partial)[J/f]$
that this sub-algebra is the image of $R(A,\partial)[J/f]$ by the
injective homogeneous localization homomorphism 
\[
R(A,\partial)[J/f]\rightarrow R(A,\partial)[J/f][f^{-1}]=R(A,\partial)_{f}.
\]
The equivariance then follows readily from the construction of the
$k$-derivations $R(\partial')$ and $R(\partial)'$.
\end{proof}
\begin{cor}
Let $A$ be a finitely generated $k$-algebra endowed with a nonzero
locally nilpotent $k$-derivation $\partial$. If $R(A,\partial)$
is a finitely generated $k$-algebra, then so is $R(A[I/f]),\partial')$
for every equivariant affine modification $A[I/f]$ of $A$.
\end{cor}

\begin{proof}
Indeed, if $R(A,\partial)$ is of finite type over $k$, then $J=IR(A,\partial)$
is a finitely generated ideal, which implies in turn that $R(A[I/f],\partial')\cong(R(A,\partial)[J/f]$
is of finite type over $k$.
\end{proof}

\subsection{Finitely generated Rees algebras}

It is well known that the ring of invariants of a $\mathbb{G}_{a}$-action
on an affine $k$-variety $X=\mathrm{Spec}(A)$ is in general not
finitely generated (see e.g. \cite[Chapter 7]{FreuBook} for a survey).
As a consequence, the Rees algebra $R(A,\partial)$ as well as the
associated graded algebra $\mathrm{gr}_{\partial}A$ of an integral
$k$-algebra of finite type $A$ endowed with a non-zero locally nilpotent
$k$-derivation are in general not finitely generated. Our aim in
this subsection is to give an algebro-geometric construction of all
pairs $(A,\partial)$ consisting of a $k$-algebra of finite type
and a locally nilpotent $k$-derivation of $A$ for which the Rees
algebra $(A,\partial)$ is finitely generated. Since normalization
is a finite morphism, it follows from the Artin-Tate lemma that a
$k$-algebra is finitely generated if and only its normalization is
finitely generated. By Lemma \ref{lem:Normalization-Commute-Rees},
we can thus restrict without loss of generality to the case of normal
$k$-algebras.
\begin{lem}
\label{lem:Equiv-charac-Rees-Finite-Gen}Let $(A,\partial)$ be an
integral normal $k$-algebra of finite type endowed with a locally
nilpotent $k$-derivation $\partial$, let $\{F_{n}\}_{n\geq0}$ be
the associated ascending filtration and let $R(A,\partial)=\bigoplus_{n\geq0}F_{n}$
be its Rees algebra. Then the following are equivalent:

1) The algebra $R(A,\partial)$ is finitely generated over $k$,

2) The associated graded algebra $\mathrm{gr}_{\partial}A$ is finitely
generated over $k$,

3) The $k$-algebra $A_{0}=F_{0}=\mathrm{ker}\partial$ is finitely
generated and $R(A,\partial)_{+}=\bigoplus_{n>0}F_{n}$ is a finitely
generated $R(A,\partial)$-module.
\end{lem}

\begin{proof}
As in (\ref{eq:Graded-inclusions}), we identify $R(A,\partial)$
with the graded sub-$A_{0}$-algebra $\bigoplus_{n\geq0}F_{n}\upsilon^{n}$
of $A[\upsilon]$. The implication 3) $\Rightarrow$ 1) is straightforward
and the implication 1) $\Rightarrow$ 2) follows immediately from
the definition of $\mathrm{gr}_{\partial}A=R(A,\partial)/\upsilon R(A,\partial)$.
To show the implication 2) $\Rightarrow$ 3), we can assume without
loss of generality that $\mathrm{gr}_{\partial}A=k[b_{1},\ldots,b_{r}]$
for some nonzero homogeneous elements $b_{i}\in F_{d(i)}/F_{d(i)-1}$,
$i=1,\ldots,r$, where $d(i)=0$ for $i=1,\ldots,m$ and $d(i)>0$
for $i=m+1,\ldots,r$. It follows in particular that $F_{0}=F_{0}/F_{-1}$
is generated by $b_{1},\ldots,b_{m}$. Choosing representatives $a_{i}\in F_{d(i)}\setminus F_{d(i)-1}$
of the classes $b_{i}$, we have $F_{0}=k[b_{1},\ldots,b_{m}]=k[a_{1},\ldots,a_{m}]$.
We claim that $R(A,\partial)_{+}$ is equal to the homogeneous ideal
$I$ generated by $\upsilon$ and the elements $a_{i}\upsilon^{d(i)}$,
$i=m+1,\ldots,r$. Indeed, let $f\upsilon^{d}\in F_{d}\upsilon^{d}\subset R(A,\partial)_{+}$
be a homogeneous element and let $d_{0}$ be minimal such that $f\in F_{d_{0}}\setminus F_{d_{0}-1}$.
If $d_{0}<d$ then $f\upsilon^{d_{0}}\in F_{d_{0}}\upsilon^{d_{0}}\subset R(A,\partial)$
and then $f\upsilon^{d}=(f\upsilon^{d_{0}})\upsilon^{d-d_{0}}\in I$.
Otherwise, if $d_{0}=d$, the residue class $\overline{f}$ of $f$
in $F_{d}/F_{d-1}$ is nonzero, and by hypothesis, there exists a
homogeneous polynomial $P\in F_{0}[t_{m+1},\ldots,t_{r}]$ of degree
$d$ with respect to the weights $d(t_{i})=d(i)$, $i=m+1,\ldots,r$,
such that $\overline{f}=P(b_{m+1},\ldots,b_{r})$. It follows that
\[
f\upsilon^{d}-P(a_{m+1}\upsilon^{d(m+1)},\ldots,a_{r}\upsilon^{d(r)})\in F_{d}\upsilon^{d}
\]
is contained in the subspace $F_{d-1}\upsilon^{d}$, hence is equal
to $(g_{d-1}\upsilon^{d-1})\upsilon$ for some element $g_{d-1}\upsilon^{d-1}\in F_{d-1}\upsilon^{d-1}$
of $R(A,\partial)$, which implies in turn that $f\upsilon^{d}$ belongs
to $I$.
\end{proof}

\subsubsection{Geometric criterion for finite generation}

Recall that a $\mathbb{P}^{1}$\emph{-fibration} between algebraic
$k$-varieties is a surjective projective morphism of finite type
$\pi:Y\rightarrow Y_{0}$ whose fiber $Y_{\eta}$ over the generic
point $\eta$ of $Y_{0}$ is isomorphic to the projective line $\mathbb{P}_{k(Y_{0})}^{1}$
over the field of rational functions $k(Y_{0})$ of $Y_{0}$. 
\begin{prop}
\label{prop:ReesAlg-P1-fib}Let $(A,\partial)$ be an integral normal
$k$-algebra of finite type endowed with a nontrivial locally nilpotent
$k$-derivation $\partial$ whose Rees algebra $R(A,\partial)=\bigoplus_{n\geq0}F_{n}\upsilon^{n}\subseteq A[\upsilon]$
is a finitely generated $k$-algebra. Let $A_{0}=F_{0}$ and let 
\[
X=\mathrm{Spec}(A)\hookrightarrow Y=\mathrm{Proj}_{k}(R(A,\partial))
\]
be the open embedding of schemes over $Y_{0}=\mathrm{Spec}(A_{0})$
induced by the graded inclusion $\eta:R(A,\partial)\hookrightarrow A[\upsilon]$
$($see (\ref{eq:Graded-inclusions})$)$. Then the following hold:

1) The schemes $Y_{0}$ and $Y$ are normal $k$-varieties,

2) The structure morphism $\pi:Y\rightarrow Y_{0}$ is a $\mathbb{P}^{1}$-fibration,

3) The effective Weil divisor $B=V_{+}(\upsilon)$ on $Y$ is ample
and the restriction of the sheaf $\mathcal{O}_{Y}(B)$ to the generic
fiber $Y_{\eta}\cong\mathbb{P}_{k(Y_{0})}^{1}$ of $\pi$ is equal
to $\mathcal{O}_{\mathbb{P}_{k(Y_{0})}^{1}}(1)$.
\end{prop}

\begin{proof}
By Lemma \ref{lem:Equiv-charac-Rees-Finite-Gen}, $A_{0}$ is a $k$-algebra
of finite type. Since $A$ is normal by assumption, the normality
of $Y_{0}=\mathrm{Spec}(A_{0})$ and $Y$ follow from \cite[Proposition 1.13]{FreuBook}
and Lemma \ref{lem:Normalization-Commute-Rees} respectively. Since
$R(A,\partial)$ is finitely generated over $A_{0}$, it follows from
\cite[Proposition 4.6.18]{EGAII} that $\pi:Y=\mathrm{Proj}_{k}(R(A,\partial))\rightarrow Y_{0}$
is a morphism of finite type and that there exists $d\geq1$ such
that the quasi-coherent $\mathcal{O}_{Y}$-module $\mathcal{O}_{Y}(d)$
associated to $F_{d}$ is invertible and $\pi$-ample. It follows
that $dB=V_{+}(\upsilon^{d})$ is a $\pi$-ample Cartier divisor,
and since $Y_{0}$ is affine, we deduce in turn from \cite[Proposition 4.5.10]{EGAII}
that $B$ is $\mathbb{Q}$-Cartier and ample on $Y$. By Lemma \ref{lem:Graded-Hom-Open-Immersion},
the image of the open embedding $X\hookrightarrow Y$ coincides with
the complement of the support of $B$ on $Y$. Furthermore, by Lemma
\ref{lem:Projective-Ga-action-extension}, the inclusion $X\hookrightarrow Y$
is equivariant for the $\mathbb{G}_{a}$-actions associated with the
locally nilpotent $k$-derivations $\partial$ and $\upsilon R(\partial)$
on $A$ and $R(A,\partial)$ respectively. Lemma \ref{lem:Rees-preslice}
implies that for every $s\in F_{1}\setminus F_{0}$, we have an isomophism
\[
(R(A,\partial)_{\partial s},\upsilon R(\partial))\cong((A_{0})_{\partial s}[s,\upsilon],\upsilon\tfrac{\partial}{\partial s}).
\]
It follows that the restriction of $\pi$ over the principal affine
open subset $(Y_{0})_{\partial s}\cong\mathrm{Spec}((A_{0})_{\partial s})$
is isomorphic to the trivial $\mathbb{P}^{1}$-bundle $\mathrm{Proj}_{k}((A_{0})_{\partial s}[s,\upsilon])\rightarrow(Y_{0})_{\partial s}$
and that the restriction of $\mathcal{O}_{Y}(B)$ over $(Y_{0})_{\partial s}$
is equal to $\mathcal{O}_{\mathbb{P}_{(A_{0})_{\partial s}}^{1}}(1)$. 
\end{proof}
Conversely, given a normal affine $k$-variety $Y_{0}$ and a $\mathbb{P}^{1}$-fibration
$\pi:Y\rightarrow Y_{0}$ where $Y$ is a normal $k$-variety, it
is a natural question to characterize which effective Weil divisors
$B$ on $Y$ have the property that their complements are affine varieties
carrying $\mathbb{G}_{a}$-actions with finitely generated associated
Rees algebras. Recall that a Weil divisor $B$ on a $k$-variety $Y$
is called \emph{semi-ample} if there exists $n\geq1$ such that the
sheaf $\mathcal{O}_{Y}(nB)$ is invertible and generated by its global
sections. We then have the following criterion:
\begin{thm}
\label{thm:Semi-ample-complement}Let $Y_{0}=\mathrm{Spec}(A_{0})$
be a normal affine $k$-variety and let $\pi:Y\rightarrow Y_{0}$
be a $\mathbb{P}^{1}$-fibration where $Y$ is a normal $k$-variety.
Let $B$ an effective semi-ample Weil divisor on $Y$ with the following
properties: 

a) The scheme $X=Y\setminus B$ is an affine $k$-variety,

b) The restriction of the sheaf $\mathcal{O}_{Y}(B)$ to the generic
fiber $Y_{\eta}\cong\mathbb{P}_{k(Y_{0})}^{1}$ of $\pi$ is equal
to $\mathcal{O}_{\mathbb{P}_{k(Y_{0})}^{1}}(1)$.

\noindent Then the following hold:

1) There exists a nontrivial $\mathbb{G}_{a,Y_{0}}$-action on $Y$
which leaves $B$ invariant and restricts to a $\mathbb{G}_{a,Y_{0}}$-action
on $X$.

2) The Rees algebra $R(A,\partial)$ of the locally nilpotent $k$-derivation
$\partial$ of $A=\Gamma(X,\mathcal{O}_{X})$ corresponding to the
induced $\mathbb{G}_{a,Y_{0}}$-action on $X$ is a finitely generated
$A_{0}$-algebra isomorphic to $R(Y,B)=\bigoplus_{n\geq0}H^{0}(Y,\mathcal{O}_{Y}(nB))$.
\end{thm}

\begin{proof}
By definition, $\mathcal{L}=\mathcal{O}_{Y}(B)$ is the reflexive
subsheaf of rank $1$ of the constant sheaf $\mathcal{K}_{Y}$ of
rational functions on $Y$ defined by 
\[
\mathcal{L}(U)=\{f\in\mathcal{K}_{Y}(U),\mathrm{div}(f)+B|_{U}\geq0\}\cup\{0\}
\]
for every open subset $U$ of $Y$. The fact that $B$ is effective
implies that the constant section $1$ of $\mathcal{K}_{Y}$ is contained
in $\mathcal{L}$. We denote by $\upsilon\in H^{0}(Y,\mathcal{L})$
the corresponding global section of $\mathcal{L}$ whose zero locus
is equal to $B$. We then get an inclusion 
\[
H^{0}(Y,\mathcal{O}_{Y})\hookrightarrow H^{0}(Y,\mathcal{L}),\;f\mapsto f\upsilon.
\]
Since $Y$ is projective over the affine variety $Y_{0}$, by \cite[Chapter III,Theorem 5.2]{Ha77},
we have $H^{0}(Y,\mathcal{O}_{Y})=A_{0}$ and $H^{0}(Y,\mathcal{L}^{\otimes n})$
is a finitely generated $A_{0}$-module for every $n$. Furthermore,
the restriction homomorphism 
\[
H^{0}(Y,\mathcal{L}^{\otimes n})\rightarrow H^{0}(Y_{\eta},\mathcal{L}^{\otimes n})=H^{0}(Y_{\eta},\mathcal{O}_{Y_{\eta}}(nB|_{Y_{\eta}}))\cong H^{0}(\mathbb{P}_{k(Y_{0})}^{1},\mathcal{O}_{\mathbb{P}_{k(Y_{0})}^{1}}(n))
\]
is surjective. The fact that for an effective semi-ample Weil divisor
$B$ the algebra $R(Y,\mathcal{L})=R(Y,B)$ is finitely generated
over $A_{0}$ is a classical result due to Zariski (see e.g. \cite{Za62}).
Let us briefly recall the argument. Since $\mathcal{L}$ is semi-ample,
it follows from \cite[Theorem 2.1.27]{LazBook} that for sufficiently
big and divisible $d\geq1$, the sheaf $\mathcal{L}^{\otimes d}$
is invertible and the rational map 
\[
\psi_{d}:Y\dashrightarrow\mathbb{P}(H^{0}(Y,\mathcal{L}^{\otimes d})^{*})
\]
is an everywhere defined morphism of $Y_{0}$-schemes with connected
fibers, whose image is a normal variety $\pi_{d}:Y_{d}\rightarrow Y_{0}$
projective over $Y_{0}$, and such that we have $\mathcal{L}^{\otimes d}=\psi_{d}^{*}\mathcal{O}_{Y_{d}}(1)$.
This implies that the Veronese subring $R(Y,\mathcal{L})^{(d)}=\bigoplus_{n\geq0}H^{0}(Y,\mathcal{L}^{\otimes nd})$
is finitely generated over $A_{0}$, and hence that $R(Y,B)=R(Y,\mathcal{L})$
is finitely generated by \cite[Corollary 1.2.5]{ADHL15}.

It follows that $Y'=\mathrm{Proj}_{k}(R(Y,B))$ is a normal variety,
projective over $Y_{0}$ and that the canonical rational map of $Y_{0}$-schemes
\[
\psi:Y\dashrightarrow Y'=\mathrm{Proj}_{k}(R(Y,B))
\]
is a morphism. Since by hypothesis the restriction of $\mathcal{L}$
to $Y_{\eta}$ is invertible and very ample, $\psi$ restricts to
an isomorphism over the generic point $\eta$ of $Y_{0}$, hence is
birational. Furthermore, since $X=Y\setminus B$ is affine hence does
not contain any complete curve, it follows that $B\cdot C>0$ for
every complete curve in $Y$ intersecting $X$. This implies that
the restriction of $\psi$ to $Y\setminus B$ is quasi-finite and
birational, hence an isomorphism onto its image by Zariski Main Theorem.
The latter coincides by construction with the complement 
\[
D_{+}(\upsilon)\cong\mathrm{Spec}(R(Y,B)/(1-\upsilon)R(Y,B))
\]
of the Weil divisor $B'=V_{+}(\upsilon)$ on $Y'$. Hypothesis b)
implies further that there exists a global section $s\in H^{0}(Y,\mathcal{L})$
different from $\upsilon$ such that 
\[
\mathrm{Proj}_{k}(R(Y,B))\times_{Y_{0}}\mathrm{Spec}(k(Y_{0}))\cong\mathrm{Proj}_{k(Y_{0})}(k(Y_{0})[s,\upsilon]).
\]
It follows that there exists $f\in A_{0}$ such that the homogeneous
locally nilpotent $k(Y_{0})$-derivation $f\upsilon\frac{\partial}{\partial s}$
of degree $0$ of $k(Y_{0})[s,\upsilon]$ extends to a homogeneous
locally nilpotent $A_{0}$-derivation $\tilde{\partial}$ of $R(Y,B)$
of degree $0$ defining a $\mathbb{G}_{a,Y_{0}}$-action on $Y'$
leaving the Weil divisor $B'$ invariant and inducing the trivial
action on $B'$. Since $\psi:Y\rightarrow Y'$ restricts to an isomorphism
$Y\setminus B\rightarrow Y'\setminus B'$, this action lifts to a
$\mathbb{G}_{a,Y_{0}}$-action on $Y$ leaving $B$ invariant and
hence $X=Y\setminus B$ invariant. By construction, the Rees algebra
of the associated locally nilpotent $A_{0}$-derivation $\partial$
of $\Gamma(X,\mathcal{O}_{X})$ is isomorphic to the finitely generated
algebra $R(Y,B)=R(Y',B')$.
\end{proof}
Given a pair $(\pi:Y\rightarrow Y_{0},B)$ satisfying the hypotheses
of Theorem \ref{thm:Semi-ample-complement}, the proof actually shows
that the composition of the open embedding $X=Y\setminus B\hookrightarrow Y$
with the canonical morphism 
\[
\psi:Y\rightarrow Y'=\mathrm{Proj}_{k}(R(Y,B))
\]
of schemes over $Y_{0}$ is an open embedding of $X$ in $Y'$ as
the complement of the ample Weil divisor $B'=\psi_{*}(B)$. The following
example illustrates the fact that even when $X$ is smooth, the variety
$Y'$ can have bad singularities supported along $B'$ so that, depending
on the context, it can be more convenient to consider a model $(Y,B)$
with better singularities but non-ample boundary divisor $B$.
\begin{example}
Let $X\subset\mathbb{A}_{k}^{4}=\mathrm{Spec}(k[x,y,u,v])$ be the
smooth affine $3$-fold with equation $xv=y(yu+1)$. The locally nilpotent
$k[x,y]$-derivation 
\[
\partial=x\tfrac{\partial}{\partial u}+y^{2}\tfrac{\partial}{\partial v}
\]
of the coordinate ring $A$ of $X$ defines a $\mathbb{G}_{a}$-action
on $X$. The ring of invariants $A_{0}$ is equal to $k[x,y]$ and
the corresponding $\mathbb{G}_{a}$-invariant morphism $\pi=\mathrm{pr}_{x,y}:X\rightarrow\mathbb{A}_{k}^{2}$
restricts to a $\mathbb{G}_{a}$-torsor over the complement of the
origin $(0,0)$. On the other hand, $\pi^{-1}((0,0))$ is isomorphic
to $\mathbb{A}_{k}^{2}=\mathrm{Spec}(k[u,v])$ and consists of $\mathbb{G}_{a}$-fixed
points only.

The Rees algebra $R(A,\partial)$ is isomorphic to the quotient of
$k[x,y][u,v,\upsilon]$ by the homogeneous ideal generated by $xv-y^{2}u-y\upsilon$,
where $u$, $v$ and $\upsilon$ all have weight $1$. So $Y'=\mathrm{Proj}_{k}(R(A,\partial))$
is isomorphic to the closed sub-variety in $\mathbb{A}_{k}^{2}\times\mathbb{P}_{k}^{2}=\mathrm{Proj}_{k[x,y]}(k[x,y][u,v,\upsilon])$
defined by the equation $xv-y^{2}u-y\upsilon=0$, and $X=Y'\setminus B'$
where $B'$ is the irreducible ample relative hyperplane section $\{\upsilon=0\}$,
isomorphic to the blow-up of $\mathbb{A}_{k}^{2}$ with center at
the closed subscheme with defining ideal $(x,y^{2})$. The projection
$\overline{\pi}=\mathrm{pr}_{x,y}:Y'\rightarrow\mathbb{A}_{k}^{2}$
restricts to a locally trivial $\mathbb{P}^{1}$-bundle over the complement
of the origin whereas the fiber $\overline{\pi}^{-1}(0,0)$ is isomorphic
to $\mathbb{P}_{k}^{2}=\mathrm{Proj}_{k}(k[u,v,\upsilon])$. The $k[x,y]$-derivation
$\partial$ extends to the homogeneous $k[x,y,\upsilon]$-derivation
$\upsilon\partial$ of degree $0$ of $k[x,y][u,v,\upsilon]$ defining
a $\mathbb{G}_{a}$-action 
\[
((x,y),[u:v:\upsilon])\mapsto((x,y),[u+tx\upsilon:v+y^{2}\upsilon:\upsilon])
\]
on $\mathbb{A}_{k}^{2}\times\mathbb{P}_{k}^{2}$, leaving $Y'$ invariant.
Its restriction to $X$ is equal to that defined by $\partial$ whereas
it restriction to $B'$ is the trivial $\mathbb{G}_{a}$-action.

It is easily seen by the Jacobian criterion that $Y'$ has a unique
singular point $p=((0,0),[1:0:0])$, which is contained in $B'$.
Let $c:Y\rightarrow Y'$ be the blow-up of the Weil divisor $D=\overline{\pi}^{-1}(0,0)$
and let $E$ be its exceptional locus. Since $D$ is $\mathbb{G}_{a}$-invariant,
the $\mathbb{G}_{a}$-action on $Y'$ lifts to a $\mathbb{G}_{a}$-action
on $Y$. Furthermore, since $Y'\setminus\{p\}$ is smooth, $D|_{Y'\setminus\{p\}}$
is a Cartier divisor, which implies that $c$ induces a $\mathbb{G}_{a}$-equivariant
isomorphism between $Y'\setminus\{p\}$ and $c^{-1}(Y'\setminus\{p\})\cong Y\setminus E$.
In particular, $c$ induces a $\mathbb{G}_{a}$-equivariant isomorphism
between $X$ and $c^{-1}(X)=Y\setminus c^{-1}(B)$.

The intersection of $Y'$ with the affine chart $V=\{u\neq0\}$ of
$\mathbb{A}_{k}^{2}\times\mathbb{P}_{k}^{2}$ is isomorphic to the
sub-variety $xv-yz=0$ in $\mathbb{A}_{k}^{4}$, where $z=y-\upsilon$.
The point $p$ is thus a non-$\mathbb{Q}$-factorial singularity of
$Y'$, the divisor $D|_{V}$ is not $\mathbb{Q}$-Cartier, and the
blow-up $c:Y\rightarrow Y'$ of $D$ is a small resolution of $p$
with exceptional locus $E\cong\mathbb{P}_{k}^{1}$. The threefold
$Y$ is thus smooth and $B=c^{-1}(B')$ is a $\mathbb{G}_{a}$-invariant
irreducible semi-ample Cartier divisor which is not ample, such that
$Y\setminus B$ is equivariantly isomorphic to $X$.
\end{example}

\subsubsection{\label{subsec:Rees-Algo}The Rees algebra algorithm}

Let $X=\mathrm{Spec}(A)$ be a normal affine variety endowed with
a nontrivial $\mathbb{G}_{a}$-action determined by a locally nilpotent
$k$-derivation $\partial$ of $A$. Let $A_{0}=F_{0}=\mathrm{Ker}\partial$
and $\{F_{n}\}_{n\geq0}$ be the associated ascending filtration of
$A$ by its $A_{0}$-submodules. In the case where $A_{0}$ is noetherian,
an algorithm to compute the modules $F_{n}$ was given by Freudenburg
\cite{Fr16,FreuBook} in the form of an extension of van den Essen's
kernel algorithm for a locally nilpotent derivation \cite[$\S$ 1.4]{vdEBook}.
In the case where the Rees algebra $R(A,\partial)$ is finitely generated,
we describe below an extension of these algorithms, which computes
generators of $R(A,\partial)$ from a given set of generators of $A$
as a $k$-algebra.

As in (\ref{eq:Graded-inclusions}), we identify $R(A,\partial)$
with the graded $A_{0}$-subalgebra $\bigoplus_{n\geq0}F_{n}\upsilon^{n}$
of $A[\upsilon]$. Let $a_{1},\ldots,a_{m}\in A$ be a finite collection
of generators of $A$ as a $k$-algebra. For every $i=1,\ldots,m$,
we choose an integer $e(i)$ so that $a_{i}\in F_{e(i)}$. We obtain
a graded subalgebra 
\[
R_{0}=k[\upsilon,\{a_{i}\upsilon^{e(i)}\}_{i=1,\ldots,m}]\subseteq R(A,\partial).
\]
If equality holds, we are done. Otherwise, there exists an element
$a\in F_{d}\setminus F_{d-1}\subset A$, for some $d\geq0$, such
that $a\upsilon^{d}\in R(A,\partial)\setminus R_{0}$. Since $a\in A$,
there exists a polynomial $\tilde{P}\in k[X_{1},\ldots,X_{m}]$ such
that $a=\tilde{P}(a_{1},\ldots,a_{m})$. Letting $P\in k[X_{0},\ldots,X_{m}]$
be the homogenization of $\tilde{P}$ with respect to the weights
$\mathbf{e}=(1,e(1),\ldots,e(m))$, we have $P(\upsilon,a_{1}\upsilon^{e(1)},\ldots,a_{m}\upsilon^{e(m)})=a\upsilon^{N}$
in $R(A,\partial)$ for some $N>d$. Let $N$ be minimal with the
property that $a\upsilon^{N}\in R_{0}$ and consider the graded homomorphism

\begin{eqnarray*} k[X_0,\dots,X_m]& \stackrel\phi\to & R(A,\partial)\\ X_0&\mapsto&\upsilon\\ X_i&\mapsto&a_i\upsilon^{e(i)},\textrm{ for } i\geq 1. 
\end{eqnarray*} Since $k[X_{0},\dots,X_{m}]$ is noetherian, the $\mathbf{e}$-homogeneous
ideal $\phi^{-1}(\upsilon R(A,\partial))\subset k[X_{0},\ldots,X_{m}]$
is finitely generated, say by elements $Q_{1},\dots,Q_{s}\in k[X_{0},\dots,X_{m}]$.
By definition, there exists $q_{i}\in A$ and integers $f(i)$ such
that 
\[
\tfrac{1}{\upsilon}Q_{i}(\upsilon,a_{1}\upsilon^{e(1)},\dots,a_{m}\upsilon^{e(m)})=q_{i}\upsilon^{f(i)}\in R(A,\partial),\textrm{ for }1\leq i\leq s.
\]
Since $N>d$, the polynomial $P$ belongs to $\phi^{-1}(\upsilon R(A,\partial))$,
and it follows that $P=\sum_{i=1}^{s}P_{i}Q_{i}$ for some $\mathbf{e}$-homogeneous
elements $P_{i}\in k[X_{0},\dots,X_{m}]$. Hence 
\[
a\upsilon^{N-1}=\sum_{i=1}^{s}q_{i}\upsilon^{f(i)}P_{i}(\upsilon,a_{1}\upsilon^{e(1)},\dots,a_{m}\upsilon^{e(m)}),
\]
 and it follows that 
\[
a\upsilon^{N-1}\in k[\upsilon,a_{1}\upsilon^{e(1)},\dots,a_{m}\upsilon^{e(m)},q_{1}\upsilon^{f(1)},\dots,q_{s}\upsilon^{f(s)}]\subseteq R(A,\partial).
\]
Thus by adding the generators $q_{i}\upsilon^{f(i)}$, $i=1,\ldots,s$,
to the previous ones, we obtain a subalgebra $R_{1}\subset R(A,\partial)$
with the property that 
\[
\min\{N\,\textrm{such that }a\upsilon^{N}\in R_{1}\}\leq\min\{N\,\textrm{such that }a\upsilon^{N}\in R_{0}\}-1.
\]
If $R(A,\partial)$ is finitely generated over $k$, say $\ensuremath{R(A,\partial)=k[g_{1}\upsilon^{m(1)},\dots,g_{l}\upsilon^{m(l)}]}$
with $g_{i}\in A$, then for each $g_{i}$ there exists a minimal
number $N_{i}$ such that $g_{i}\upsilon^{N_{i}}\in R_{0}\subset R(A,\partial)$.
By iterating the above procedure at most $M=\max_{1\leq i\leq l}\{N_{i}-m(i)\}$
times, we obtain a finitely generated subalgebra $R_{M}\subseteq R(A,\partial)$
which contains all the $g_{i}\upsilon^{m(i)}$, $i=1,\ldots,l$, hence
is equal to $R(A,\partial)$. 

\subsection{Relation between global Rees algebras and relative Rees algebras
of the fixed point free locus}

Let $X=\mathrm{Spec}(A)$ be a normal affine $k$-variety endowed
with a nontrivial $\mathbb{G}_{a}$-action determined by a locally
nilpotent $k$-derivation $\partial$ of $A$. Let $X^{\mathbb{G}_{a}}$
denote the fixed locus of this $\mathbb{G}_{a}$-action. By \cite[10.4]{LMB00}
the induced $\mathbb{G}_{a}$-action $\mu$ on $Y=X\setminus X^{\mathbb{G}_{a}}$
admits a categorical quotient in the category of algebraic spaces
in the form of an \'etale locally trivial $\mathbb{G}_{a}$-torsor
$\rho:Y\rightarrow S$ over a certain algebraic $k$-space $S$. Let
$\{\mathcal{F}_{n}\}_{n\geq0}$ be the filtration of $\rho_{*}\mathcal{O}_{Y}$
associated to the locally nilpotent $\mathcal{O}_{S}$-derivation
$\delta_{S}$ of $\rho_{*}\mathcal{O}_{Y}$ corresponding to the action
$\mu$. By Proposition \ref{prop:Ga-torsors-Rees}, $\mathcal{F}_{1}$
is an \'etale locally free sheaf of rank $2$ on $S$, and the Rees
$\mathcal{O}_{S}$-algebra $\mathcal{R}(Y,\mu)$ is isomorphic to
the symmetric algebra $\mathrm{Sym}^{.}\mathcal{F}_{1}$ of $\mathcal{F}_{1}$.
The relative spectrum $p:V=\mathrm{Spec}_{S}(\mathcal{R}(Y,\mu))\rightarrow S$
is thus an \'etale locally trivial vector bundle of rank $2$ on
$S$.
\begin{lem}
\label{lem:Local-vs-Global-codim2}With the above notation, suppose
that every irreducible component of the fixed locus $X^{\mathbb{G}_{a,k}}$
has codimension at least $2$ in $X$. Then $R(A,\partial)\cong\Gamma(V,\mathcal{O}_{V})$
as graded algebras.
\end{lem}

\begin{proof}
Since $X\setminus Y=X^{\mathbb{G}_{a}}$ has codimension at least
$2$ in the normal affine variety $X$, we have 
\[
\Gamma(S,\rho_{*}\mathcal{O}_{Y})=\Gamma(Y,\mathcal{O}_{Y})=\Gamma(X,\mathcal{O}_{X})=A.
\]
Furthermore, since $\mu$ is the restriction to $Y$ of the $\mathbb{G}_{a}$-action
determined by $\partial$, for every $n\geq0$, the subspaces $F_{n}=\mathrm{Ker}\partial^{n+1}$
of $A$ and $\Gamma(S,\mathcal{F}_{n})$ of $\Gamma(S,\rho_{*}\mathcal{O}_{Y})$
coincide. Indeed, by definition $\Gamma(S,\delta_{S})$ and $\partial$
extend to the same derivation $\tilde{\partial}$ of the field of
rational functions $\mathrm{Frac}(A)$ of $X$. It is clear that $F_{n}\subseteq\Gamma(S,\mathcal{F}_{n})$
and that conversely every element $f\in\Gamma(S,\mathcal{F}_{n})$
is a rational function on $X$, defined everywhere except maybe on
$X^{\mathbb{G}_{a}}$, and with the property that $\Gamma(S,\delta_{S})^{n+1}f=\tilde{\partial}^{n+1}f=0$.
Since $X^{\mathbb{G}_{a}}$ has codimension at least $2$ and $X$
is normal, $f$ is everywhere defined on $X$ and satisfies $\partial^{n+1}f=0$.
So $f$ is an element of $F_{n}$. We thus obtain isomorphisms of
graded algebras 
\begin{align*}
R(A,\partial) & =\bigoplus_{n\geq0}F_{n}=\bigoplus_{n\geq0}\Gamma(S,\mathcal{F}_{n})\cong\Gamma(S,\bigoplus_{n\geq0}\mathcal{F}_{n})\\
 & \cong\Gamma(S,\bigoplus_{n\geq0}\mathrm{Sym}^{.}\mathcal{F}_{1})=\Gamma(S,p_{*}\mathcal{O}_{V})=\Gamma(V,\mathcal{O}_{V}).
\end{align*}
\end{proof}
Let $\{F_{n}\}_{n\geq0}$ be the increasing filtration of $A$ associated
to $\partial$ and let $A_{0}=F_{0}=\mathrm{Ker}\partial$. Then for
every $n\geq1$, we have an exact sequence of $A_{0}$-modules 
\[
0\rightarrow A_{0}\rightarrow F_{n}\stackrel{\partial}{\longrightarrow}F_{n-1}
\]
in which the last homomorphism is in general not surjective. In contrast,
for the $\mathbb{G}_{a}$-torsor $\rho:Y\rightarrow S$, the sequences
of $\mathcal{O}_{S}$-module homomorphisms
\[
0\rightarrow\mathcal{O}_{S}\rightarrow\mathcal{F}_{n}\stackrel{\delta_{S}}{\longrightarrow}\mathcal{F}_{n-1}\rightarrow0,\quad n\geq1,
\]
are all exact. Suppose as in Lemma \ref{lem:Local-vs-Global-codim2}
that each irreducible component of $X\setminus Y=X^{\mathbb{G}_{a}}$
has codimension at least two. Taking global sections over $S$ in
the above exact sequence, we obtain for every $n\geq1$ a long exact
sequence 
\[
0\rightarrow\Gamma(S,\mathcal{O}_{S})=A_{0}\rightarrow\Gamma(S,\mathcal{F}_{n})=F_{n}\stackrel{\Gamma(S,\delta_{S})=\partial}{\longrightarrow}\Gamma(S,\mathcal{F}_{n-1})=F_{n-1}\stackrel{d_{1,n}}{\rightarrow}H_{\mathrm{\acute{e}t}}^{1}(S,\mathcal{O}_{S})\rightarrow\cdots
\]
in which the coboundary homomorphism $d_{1,n}:F_{n-1}\rightarrow H_{\mathrm{\acute{e}t}}^{1}(S,\mathcal{O}_{S})$
maps the constant section $1$ to the isomorphism class of the $\mathbb{G}_{a,k}$-torsor
$\rho:Y\rightarrow S$ in $H_{\mathrm{\acute{e}t}}^{1}(S,\mathcal{O}_{S})$.
This provides a cohomological interpretation of the lack surjectivity
of the homomorphism $\partial:F_{n}\rightarrow F_{n-1}$ together
with an identification $\mathrm{Im}(\partial|_{F_{n}})=\mathrm{Ker}(d_{1,n})$.
\begin{example}
\label{subsec:Rees-algebra-SL2} Let 
\[
\mathrm{SL}_{2}=\left\{ M=\left(\begin{array}{cc}
x & u\\
y & v
\end{array}\right)\in\mathcal{M}_{2}(k),\,\det M=1\right\} \cong\mathrm{Spec}(k[x,y,u,v]/(xv-yu-1)).
\]
The projection $f=\mathrm{pr}_{x,y}:\mathrm{SL}_{2}\rightarrow S=\mathbb{A}_{k}^{2}\setminus\{(0,0)\}$
is a $\mathbb{G}_{a}$-torsor for the $\mathbb{G}_{a,S}$-action $\mu$
defined by right multiplication with unipotent upper triangular matrices.
Let $\left\{ \mathcal{F}_{n}\right\} _{n\geq0}$ be the corresponding
ascending filtration of $\mathcal{A}=f_{*}\mathcal{O}_{\mathrm{SL}_{2}}$.
By Proposition \ref{prop:Ga-torsors-Rees}, $\mathcal{F}_{1}$ is
a locally free sheaf of rank $2$ on $S$, and the Rees algebra $\mathcal{R}(\mathrm{SL}_{2},\mu)=\bigoplus_{n\geq0}\mathcal{F}_{n}$
is isomorphic the symmetric algebra of $\mathcal{F}_{1}$. As a consequence
of \cite[Corollary 4.1.1]{Hor64}, $\mathcal{F}_{1}$ is equal to
the restriction to $S$ of a locally free sheaf $\mathcal{E}$ of
rank $2$ on $\mathbb{A}_{k}^{2}$, and since the latter is free by
virtue of \cite{Ses58}, it follows that $\mathcal{F}_{1}\cong\mathcal{O}_{S}^{\oplus2}$.

Explicitly, since the $\mathbb{G}_{a,S}$-torsor $f:\mathrm{SL}_{2}\rightarrow S$
becomes trivial on the cover of $S$ by the principal affine open
subsets $S_{x}=\mathrm{Spec}(k[x^{\pm1},y])$ and $S_{y}=\mathrm{Spec}(k[x,y^{\pm1}])$,
with equivariant trivializations 
\[
\mathrm{SL}_{2}|_{S_{x}}\cong S_{x}\times\mathrm{Spec}(k[x^{-1}u])\quad\textrm{and}\quad\mathrm{SL}_{2}|_{S_{y}}\cong S_{y}\times\mathrm{Spec}(k[y^{-1}v]),
\]
it follows that the sub-$\mathcal{O}_{S}$-module $\mathcal{F}_{1}$
of $\mathcal{A}$ is the extension of $\mathcal{O}_{S}$ by itself
with trivializations $\mathcal{F}_{1}|_{S_{x}}\cong\mathcal{O}_{S_{x}}\cdot x^{-1}u\oplus\mathcal{O}_{S_{x}}$
and $\mathcal{F}_{1}|_{S_{y}}\cong\mathcal{O}_{S_{y}}\cdot y^{-1}v\oplus\mathcal{O}_{S_{y}}$
and transition matrix 
\[
M=\left(\begin{array}{cc}
1 & -x^{-1}y^{-1}\\
0 & 1
\end{array}\right)\in\mathrm{GL}_{2}\left(k[x^{\pm1},y^{\pm1}]\right).
\]
Since $M$ is equal to the product 
\[
M=M_{x}\cdot M_{y}=\left(\begin{array}{cc}
x^{-1} & 0\\
-y & x
\end{array}\right)\cdot\left(\begin{array}{cc}
x & -y^{-1}\\
y & 0
\end{array}\right)
\]
where $M_{x}\in\mathrm{GL}_{2}(k[x^{\pm1},y])$ and $M_{y}\in\mathrm{GL}_{2}(k[x,y^{\pm1}])$,
we see that $\mathcal{F}_{1}$ is equal to the free sub-$\mathcal{O}_{S}$-module
of $\mathcal{A}$ generated by $u$ and $v$, so that $\mathcal{R}(\mathrm{SL}_{2},\mu)\cong\mathcal{O}_{S}[u,v]$.

Composing with the structure map $S\rightarrow\mathrm{Spec}(k)$,
we view $X=\mathrm{SL}_{2}$ as the normal affine $k$-variety with
$\mathbb{G}_{a}$-action associated to the locally nilpotent $k$-derivation
$\partial=x\frac{\partial}{\partial u}+y\frac{\partial}{\partial v}$
of its coordinate ring $A=k[x,y,u,v]/(xv-yu-1)$. The associated filtration
$\left\{ F_{n}\right\} _{n\geq0}$ is given by $A_{0}=F_{0}=\mathrm{Ker}\partial=k[x,y]$
and 
\begin{align*}
F_{n} & =\mathrm{Ker}\partial^{n+1}=\sum_{p+q=n}A_{0}\cdot u^{p}v^{q}+F_{n-1},\quad n\geq1.
\end{align*}
The Rees algebra $R(A,\partial)$ is thus equal to the quotient of
the polynomial ring $A_{0}[u,v,\upsilon]$, endowed with the grading
given by the weights $(\omega_{u},\omega_{v},\omega_{\upsilon})=(1,1,1)$,
by the principal homogeneous ideal generated by $xv-yu-\upsilon$,
hence to the polynomial ring $A_{0}[u,v]=\Gamma(S,\mathcal{R}(\mathrm{SL}_{2},\mu))$.
\end{example}

\section{Examples and Applications}

In this section, we first illustrate the computation and geometric
properties of Rees algebras on a series of classical examples in the
study of additive group actions on affine varieties, with a particular
focus on the interplay between the relative and absolute Rees algebras
and the construction of vector bundles of rank two on certain geometric
quotients. We then consider an application of Rees algebras to the
construction of families of affine extensions of $\mathbb{G}_{a}$-torsors
over punctured smooth surfaces.

\subsection{\label{subsec:Danielewski-hypersurfaces}Danielewski hypersurfaces
in $\mathbb{A}_{k}^{3}$}

Given a polynomial $P\in k[x,y]$ such that $P(0,y)$ is non-constant,
with simple roots, and an integer $n\geq1$, we let $S_{n,P}$ be
the smooth surface in $\mathbb{A}_{k}^{3}=\mathrm{Spec}(k[x,y,z])$
with equation $x^{n}z=P(x,y)$. For every nonzero polynomial $q(x)\in k[x]$,
the surface $S_{n,P}$ is equipped with a nontrivial $\mathbb{G}_{a}$-action
$\mu$ associated to the locally nilpotent $k[x]$-derivation 
\[
q(x)\partial_{n,P}=q(x)(x^{n}\tfrac{\partial}{\partial y}+\frac{\partial P}{\partial y}(x,y)\tfrac{\partial}{\partial z})
\]
of its coordinate ring $A_{n,P}$. Letting $d=\deg_{y}(P)$, the corresponding
ascending filtration of $A_{n,P}$ is given by $A_{0}=F_{0}=k[x]$
and 
\[
F_{dn+k}=k[x]\cdot y^{dn+k}+k[x]\cdot z^{n}+F_{dn+k-1}.
\]
Let $k[x][y,\upsilon,z]$ be endowed with the grading given by the
weights $(\omega_{x},\omega_{y},\omega_{\upsilon},\omega_{z})=(0,1,1,d)$
and let $\tilde{P}(x,y,\upsilon)\in k[x][y,\upsilon]$ be the unique
homogeneous polynomial with respect to the induced grading such that
$P(x,y)=\tilde{P}(x,y,1)$. The Rees algebra $R(A_{n,P},\partial_{n,P})$
is isomorphic to the quotient of $A_{0}[y,\upsilon,z]$ by the principal
homogeneous ideal generated by $x^{n}z-\tilde{P}(x,y,\upsilon)$.
So $\mathrm{Proj}_{k}(R(A_{n,P},\partial_{n,P}))$ is isomorphic to
the closed sub-scheme $\overline{S}_{n,P}$ of $\mathbb{A}_{k}^{1}\times\mathbb{P}(1,1,d)=\mathrm{Proj}_{k}(A_{0}[\upsilon,y,z])$
with equation $x^{n}z-\tilde{P}(x,y,\upsilon)=0$, in which $S_{n,P}$
embeds as the complement of the relative hyperplane section 
\[
\{\upsilon=0\}\cap\overline{S}_{n,P}=\left\{ x^{n}z-P(x,y,0)=0\right\} .
\]
The fiber of $\mathrm{pr}_{x}:\overline{S}_{n,P}\rightarrow\mathbb{A}_{k}^{1}$
over the origin is equal to the union of $\deg(\tilde{P}(0,y,\upsilon))$
copies of the projective line $\mathbb{P}_{k}^{1}$ all intersecting
at the point $[0:0:1]\in\mathbb{P}(1,1,d)$.

Since $P(0,y)$ has simple roots, the $\mathbb{G}_{a}$-action associated
to $\partial_{n,P}$ is fixed point free and the $\mathbb{G}_{a}$-invariant
projection $\mathrm{pr}_{x}:S_{n,P}\rightarrow\mathbb{A}_{k}^{1}$
factors through a $\mathbb{G}_{a}$-torsor $\rho:S_{n,P}\rightarrow\breve{\mathbb{A}}_{k}^{1}$
over the irreducible non-separated curve $\delta:\breve{\mathbb{A}}_{k}^{1}\rightarrow\mathbb{A}_{k}^{1}$
obtained from $\mathbb{A}_{k}^{1}=\mathrm{Spec}(k[x])$ by replacing
the origin $\{0\}$ by as many disjoint copies as there are irreducible
components in the fiber $\mathrm{pr}_{x}^{-1}(\{0\})\cong\mathrm{Spec}(k[y,z]/(P(0,y)))$
\cite{Du05,DuPo09}. We can thus consider the action $\mu$ as being
given by a locally nilpotent $\mathcal{O}_{\breve{\mathbb{A}}_{k}^{1}}$-derivation
$\breve{\partial}$ of $\rho_{*}\mathcal{O}_{S_{n,P}}$. By Proposition
\ref{prop:Ga-torsors-Rees}, the Rees algebra $\mathcal{R}(S_{n,P},\mu)$
is then canonically isomorphic to the symmetric algebra $\mathrm{Sym}^{\cdot}\mathcal{F}_{1}$
of the locally free sheaf $\mathcal{F}_{1}=\mathcal{K}er\breve{\partial}^{2}$
of rank $2$ on $\breve{\mathbb{A}}_{k}^{1}$ which fits in the exact
sequence 
\[
0\rightarrow\mathcal{O}_{\breve{\mathbb{A}}_{k}^{1}}\rightarrow\mathcal{F}_{1}\stackrel{\breve{\partial}}{\rightarrow}\mathcal{O}_{\breve{\mathbb{A}}_{k}^{1}}\rightarrow0.
\]
By Lemma \ref{lem:Local-vs-Global-codim2}, we have $R(A_{n,P},\partial_{n,P})\cong\Gamma(V,\mathcal{O}_{V})$
where $p:V=\mathrm{Spec}_{\breve{\mathbb{A}}_{k}^{1}}(\mathrm{Sym}^{\cdot}\mathcal{F}_{1})\rightarrow\breve{\mathbb{A}}_{k}^{1}$
is the vector bundle of rank $2$ on $\breve{\mathbb{A}}_{k}^{1}$
determined by $\mathcal{F}_{1}$.

If $\deg P(0,y)\geq2$, then 
\[
R(A_{n,P},\partial_{n,P})\otimes_{A_{0}}(A_{0}/(x))\cong k[y,\upsilon,z]/(\tilde{P}(0,y,\upsilon))
\]
is not a polynomial ring in two variables over $k=A_{0}/(x)$. So
$R(A_{n,P},\partial_{n,P})$ is not isomorphic to a polynomial ring
in two variables over $A_{0}$, which implies that $V$ is a nontrivial
vector bundle over $\breve{\mathbb{A}}_{k}^{1}$, since otherwise
$\Gamma(V,\mathcal{O}_{V})\cong R(A_{n,P},\partial_{n,P})$ would
be isomorphic to a polynomial ring in two variables over $\Gamma(\breve{\mathbb{A}}_{k}^{1},\mathcal{O}_{\breve{\mathbb{A}}_{k}^{1}})=A_{0}$.

Otherwise, if $\deg P(0,y)=1$, then $\mathrm{pr}_{x}:S_{n,P}\rightarrow\mathbb{A}_{k}^{1}=\mathrm{Spec}(A_{0})$
is a $\mathbb{G}_{a}$-torsor for the $\mathbb{G}_{a}$-action determined
by the $k$-derivation $\partial_{n,P}$, hence is the trivial one
since $\mathrm{Spec}(A_{0})$ is affine. It follows in turn that $V$
is the trivial rank $2$ vector bundle on $\mathrm{Spec}(A_{0})$.
In the special case where $P$ is equal to the constant polynomial
$y$, $S_{n,P}$ is isomorphic to $\mathbb{A}_{k}^{2}=\mathrm{Spec}(k[x,z])$
and the $\mathbb{G}_{a}$-action given by $q(x)\partial_{n,P}$ concides
with that given by the locally nilpotent $k[x]$-derivation $\partial=q(x)\frac{\partial}{\partial z}$.
It is a classical result \cite{Ren68} that every $\mathbb{G}_{a}$-action
on $\mathbb{A}_{k}^{2}$ is conjugate to an action defined by a locally
nilpotent derivation of this form. The corresponding filtration is
given by $A_{0}=F_{0}=k[x]$ and 
\begin{align*}
F_{1} & =k[x]y\oplus k[x]\cdot1\\
F_{n} & =k[x]\cdot y^{n}\oplus F_{n-1}\cong\mathrm{Sym}_{F_{0}}^{n}F_{1},\quad\textrm{for }n\geq2
\end{align*}
so that we have an isomorphism of graded $k[x]$-algebras $R(k[x,z],\partial)\cong k[x][y,\upsilon]$,
where $\upsilon$ and $y$ both have homogeneous degree $1$.

\subsection{\label{exa:Ga-extension-preEx} A smooth affine threefold whose geometric
quotient is quasi-projective but not quasi-affine}

It is known in general that the geometric quotient $X/\mathbb{G}_{a}$
of a proper $\mathbb{G}_{a}$-action on a factorial affine variety
is a quasi-affine variety (see e.g. \cite{DeFiGe94}). For smooth
affine threefolds $X$, factorial or not, it is a consequence of Chow's
Lemma that the algebraic space geometric quotient $X/\mathbb{G}_{a}$
of a proper $\mathbb{G}_{a}$-action is a quasi-projective surface.
In this subsection, we consider a simple example of a smooth non-factorial
affine threefold $X$ endowed with a proper $\mathbb{G}_{a}$-action,
whose geometric quotient is a smooth quasi-projective surface which
is not quasi-affine.

Let $S_{0}=\mathrm{Spec}(k[x,y])$ and let $X\subset\mathbb{A}_{S_{0}}^{3}=\mathrm{Spec}(k[x,y][w_{1},w_{2},w_{3}])$
be the smooth threefold defined by the system of equations
\[
\begin{cases}
yw_{3}-w_{1}w_{2} & =0\\
xw_{2}-y(yw_{1}+1) & =0\\
xw_{3}-w_{1}(yw_{1}+1) & =0.
\end{cases}
\]
The threefold $X$ can be endowed with a fixed point free $\mathbb{G}_{a,S_{0}}$-action
$\mu:\mathbb{G}_{a,S_{0}}\times_{S_{0}}X\rightarrow X$ induced by
the locally nilpotent $k[x,y]$-derivation 
\[
\partial=x\tfrac{\partial}{\partial w_{1}}+y^{2}\tfrac{\partial}{\partial w_{2}}+(2yw_{1}+1)\tfrac{\partial}{\partial w_{3}}
\]
of its coordinate ring $A$. The ring of invariants $A_{0}=\mathrm{Ker}\partial$
is equal to $k[x,y]$. The Rees algebra $R(A,\partial)$ is isomorphic
to the quotient of the polynomial ring $A_{0}[\upsilon,W_{1},W_{2},W_{3}]$
in four variables over $A_{0}$ with weights $(\omega_{\upsilon},\omega_{W_{1}},\omega_{W_{2}},\omega_{W_{3}})=(1,1,1,2)$
by the homogeneous ideal $I$ generated by the polynomials $xW_{2}-y(yW_{1}+\upsilon)$,
$yW_{3}-W_{1}W_{2}$ and $xW_{3}-W_{1}(yW_{1}+\upsilon)$. So $\mathrm{Proj}_{k}(R(A,\partial))$
is isomorphic to the closed sub-scheme $Y$ of 
\[
\mathbb{A}_{k}^{2}\times\mathbb{P}(1,1,1,2)=\mathrm{Proj}_{A_{0}}(A_{0}[\upsilon,W_{1},W_{2},W_{3}])
\]
defined by the vanishing of these polynomials. The threefold $X$
embeds in $Y$ as the complement of the relative hyperplane section
$V_{+}(\upsilon)\cap Y\subset\mathbb{P}(1,1,2)$, defined by the equations
$xW_{2}-y^{2}W_{1}=0$, $yW_{3}-W_{1}W_{2}=0$ and $xW_{3}-yW_{1}^{2}=0$.

It is easily seen from this description that the restriction of $\overline{\pi}=\mathrm{pr}_{\mathbb{A}_{k}^{2}}:Y\rightarrow S_{0}=\mathbb{A}_{k}^{2}$
over the complement of the origin $o=\{(0,0)\}$ is a Zariski locally
trivial $\mathbb{P}^{1}$-bundle having $\{\upsilon=0\}$ as a section,
so that the restriction of the algebraic quotient morphism $\pi=\mathrm{pr}_{x,y}:X\rightarrow S_{0}$
over $S_{0}\setminus\{o\}$ is a $\mathbb{G}_{a}$-torsor $P\rightarrow S_{0}\setminus\{o\}$.
More explicitly, letting $S_{0,x}=\mathrm{Spec}(k[x^{\pm1},y])$ and
$S_{0,y}=\mathrm{Spec}(k[x,y^{\pm1}])$, we have $\mathbb{G}_{a}$-equivariant
isomorphisms 
\[
\begin{cases}
\pi^{-1}(S_{0,x})\cong\mathrm{Spec}(k[x^{\pm1},y][x^{-1}w_{1}]) & \cong S_{0,x}\times\mathbb{G}_{a}\\
\pi^{-1}(S_{0,y})\cong\mathrm{Spec}(k[x,y^{\pm1}][-y^{-2}w_{2}]) & \cong S_{0,y}\times\mathbb{G}_{a}
\end{cases}
\]
where $\mathbb{G}_{a}$ acts on $S_{0,x}\times\mathbb{G}_{a}$ and
$S_{0,y}\times\mathbb{G}_{a}$ by translations on the second factor.
On the other hand, the scheme-theoretic fiber $\overline{\pi}^{-1}(o)\subset\mathbb{P}(1,1,1,2)$
is the union of the surface $\{W_{1}=0\}\cong\mathbb{P}(1,1,2)$ and
the line $\{\upsilon=W_{2}=0\}$. It follows that $\pi^{-1}(o)=\overline{\pi}^{-1}(o)\setminus V_{+}(\upsilon)$
is isomorphic to $\mathbb{A}_{k}^{2}=\mathrm{Spec}(k[w_{2},w_{3}])$.
In particular the class group of $X$ is isomorphic to $\mathbb{Z}$,
generated by the class of the divisor $\pi^{-1}(o)$, and the algebraic
quotient morphism $\pi:X\rightarrow S_{0}$ is not the geometric quotient
of the fixed point free action $\mu$ on $X$. Let 
\[
\tau=\mathrm{pr}_{S_{0}}:S_{1}=\{xv-yu=0\}\subset S_{0}\times\mathbb{P}_{[u:v]}^{1}\rightarrow S_{0}
\]
be the blow-up of the origin $o\in S_{0}$. The morphism $\pi$ lifts
to a morphism 
\[
\tilde{\pi}:X\rightarrow S_{1},\quad(x,y,w_{1},w_{2},w_{3})\mapsto((x,y),[yw_{1}+1:w_{2}]).
\]
which maps $\pi^{-1}(o)$ dominantly onto the exceptional divisor
$E\cong\tau^{-1}(o)\cong\mathrm{Proj}(k[u,v])$ of $\tau$.
\begin{lem}
The morphism $\tilde{\pi}:X\rightarrow S_{1}$ factors through a Zariski
locally trivial $\mathbb{G}_{a}$-torsor $\rho:X\rightarrow S$ over
the smooth quasi-projective but not quasi-affine surface $S=S_{1}\setminus\{o_{1}\}$,
where $o_{1}=((0,0),[0:1])\in E\subset S_{1}$.
\end{lem}

\begin{proof}
The induced $\mathbb{G}_{a}$-action on $\pi^{-1}(o)\cong\mathrm{Spec}(k[w_{2},w_{3}])$
is the translation defined by the locally nilpotent $k$-derivation
$\partial_{w_{3}}$ of $k[w_{2},w_{3}]$. The morphism $\tilde{\pi}:X\rightarrow S_{1}$
is $\mathbb{G}_{a}$-invariant and the induced morphism
\[
\tilde{\pi}|_{\pi^{-1}(o)}:\pi^{-1}(o)=\mathrm{Spec}(k[w_{2},w_{3}])\rightarrow E,\quad(w_{2},w_{3})\mapsto[1:w_{2}]
\]
factors as the composition of the geometric quotient $\pi^{-1}(o)\rightarrow\pi^{-1}(o)/\mathbb{G}_{a}\simeq\mathrm{Spec}(k[w_{2}])$
with the open immersion $\pi^{-1}(o)/\mathbb{G}_{a}\hookrightarrow E$
of $\pi^{-1}(o)/\mathbb{G}_{a}$ as the complement of the point $o_{1}=((0,0),[0:1])\in E$.
It follows that $\tilde{\pi}$ factors through a surjective morphism
$\rho:X\rightarrow S=S_{1}\setminus\{o_{1}\}$ whose fibers all consist
of precisely one $\mathbb{G}_{a}$-orbit. Since $\rho$ is a smooth
morphism, it is thus a $\mathbb{G}_{a}$-torsor. By construction $S$
is smooth and quasi-projective and $\tau|_{S}:S\rightarrow B$ induces
an isomorphism $(\tau|_{S})^{*}:A_{0}\rightarrow\Gamma(S,\mathcal{O}_{S})$.
If $S$ was quasi-affine, then $\tau|_{S}:S\rightarrow S_{0}=\mathrm{Spec}(\Gamma(S,\mathcal{O}_{S}))$
would be an open immersion, which is impossible since $\tau|_{S}$
contracts $E\cap S\cong\mathbb{A}_{k}^{1}$ to the point $o\in S_{0}$.
\end{proof}
By Proposition \ref{prop:Ga-torsors-Rees}, the Rees $\mathcal{O}_{S}$-algebra
$\mathcal{R}(X,\mu)$ is equal to the symmetric algebra of a Zariski
locally free sheaf $\mathcal{F}_{1}$ of rank $2$ on $S$. The corresponding
rank $2$ vector bundle $p:V=\mathrm{Spec}_{S}(\mathrm{Sym}^{\cdot}\mathcal{F}_{1})\rightarrow S$
is nontrivial. Indeed otherwise $\Gamma(V,\mathcal{O}_{V})$ would
be isomorphic to a polynomial ring in two variables over $\Gamma(S,\mathcal{O}_{S})\cong A_{0}$
but on the other hand it follows from Lemma \ref{lem:Local-vs-Global-codim2}
and the description above that 
\[
\Gamma(V,\mathcal{O}_{V})\cong R(A,\partial)\cong A_{0}[\upsilon,W_{1},W_{2},W_{3}]/I
\]
is not a polynomial ring in two variables over $A_{0}$.

\subsection{A triangular $\mathbb{G}_{a}$-action on $\mathbb{A}_{k}^{3}$ and
the Russell cubic threefold}

Let $\mathbb{A}_{k}^{4}=\mathrm{Spec}(k[x,y,z,t])$ be endowed with
the $\mathbb{G}_{a}$-action defined by the triangular locally nilpotent
$k[x,t]$-derivation 
\[
\Delta=x^{2}\tfrac{\partial}{\partial y}+2y\tfrac{\partial}{\partial z}.
\]
The kernel $A_{0}=F_{0}$ of $\Delta$ is equal to $k[x,t,w]$ where
$w=x^{2}z-y^{2}$, and for every $n=i+2j\geq1$, we have 
\[
F_{n}=\mathrm{Ker}\Delta^{n+1}=\sum_{i+2j=n}A_{0}y^{i}z^{j}+F_{n-1}.
\]
The Rees algebra $R(k[x,y,z,t],\Delta)$ is isomorphic to the quotient
of the polynomial ring $A_{0}[y,\upsilon,z]$, where $y$, $\upsilon$
and $z$ have homogenenous degrees $1$, $1$ and $2$ respectively,
by the homogeneous ideal generated by $x^{2}z-y^{2}-w\upsilon^{2}$.

The $\mathbb{G}_{a}$-action on $\mathbb{A}_{k}^{4}$ defined by $\Delta$
is fixed point free outside the $\mathbb{G}_{a}$-invariant plane
\[
P=\{x=y=0\}\cong\mathrm{Spec}(k[z,t]).
\]
The induced $\mathbb{G}_{a}$-action $\mu$ on the quasi-affine fourfold
$\mathbb{A}_{k}^{4}\setminus P$ admits a geometric quotient in the
category of algebraic spaces in the form of an \'etale locally trivial
$\mathbb{G}_{a}$-torsor $\rho:\mathbb{A}_{k}^{4}\setminus P\rightarrow S$
over an algebraic space $S=(\mathbb{A}_{k}^{4}\setminus P)/\mathbb{G}_{a}$.
The latter is the $\mathbb{A}^{1}$-cylinder $\mathfrak{S}\times\mathrm{Spec}(k[t])$
over a $2$-dimensional smooth algebraic space of finite type
\[
\delta:\mathfrak{S}\rightarrow\mathbb{A}_{k}^{2}\setminus\{(0,0)\}=\mathrm{Spec}(k[x,w])\setminus\{(0,0)\}
\]
obtained from $\mathbb{A}_{k}^{2}\setminus\{(0,0\}$ by replacing
the curve $\{x=0\}\cong\mathrm{Spec}(k[w^{\pm1}])$ by the total space
of the \'etale double cover $\mathrm{pr}_{w}:\mathrm{Spec}(k[y,w^{\pm1}]/(y^{2}+w))\cong\mathrm{Spec}(k[y^{\pm1}])\rightarrow\mathrm{Spec}(k[w^{\pm1}])$
(see e.g. \cite{Du-ExQuot19} and the references therein). Let $\left\{ \mathcal{F}_{n}\right\} _{n\geq0}$
be the ascending filtration of $\mathcal{\rho}_{*}\mathcal{O}_{\mathbb{A}_{k}^{4}\setminus P}$
associated to the $\mathbb{G}_{a,S}$-action $\mu$. Since $\rho:\mathbb{A}_{k}^{4}\setminus P\rightarrow S$
is a $\mathbb{G}_{a,S}$-torsor, it follows from Proposition \ref{prop:Ga-torsors-Rees}
that the Rees $\mathcal{O}_{S}$-algebra $\mathcal{R}(\mathbb{A}_{k}^{4}\setminus P,\mu)$
is equal to the symmetric algebra of the rank $2$ \'etale locally
free sheaf $\mathcal{F}_{1}$ on $S$. Let $p:V=\mathrm{Spec}_{S}(\mathrm{Sym}^{\cdot}\mathcal{F}_{1})\rightarrow S$
be the corresponding vector bundle. Since $P$ has pure codimension
$2$ in $\mathbb{A}_{k}^{4}$, we have $\Gamma(V,\mathcal{O}_{V})=R(k[x,y,z,t],\Delta)$
by Lemma \ref{lem:Local-vs-Global-codim2}. Since $R(k[x,y,z,t],\Delta)$
is not isomorphic to a polynomial ring in two variables over $\Gamma(S,\mathcal{O}_{S})=A_{0}$,
it follows that $p:V\rightarrow S$ is a nontrivial vector bundle.\\

The closed subsets $X_{1}$ and $X_{3}$ of $\mathbb{A}_{k}^{4}$
with equations $x^{2}z=y^{2}+t$ and $x^{2}z=y^{2}+t^{3}+x$ are $\mathbb{G}_{a}$-invariant,
respectively isomorphic to $\mathbb{A}_{k}^{3}=\mathrm{Spec}(k[x,y,z])$
and the Russell cubic threefold \cite{KaML97}. The restrictions of
$\rho:\mathbb{A}_{k}^{4}\setminus P\rightarrow S$ to $X_{1}\setminus P=\mathrm{Spec}(k[x,y,z])\setminus\{x=y=0\}$
and $X_{3}\setminus P=X_{3}\setminus\{x=y=t=0\}$ are $\mathbb{G}_{a}$-torsors
over the closed subspaces $S_{1}$ and $S_{3}$ of $S$ whose ideal
sheaves are generated by $w-t$ and $w-(t^{3}+x)$ respectively. These
two spaces are isomorphic to $\mathfrak{S}$ \cite[Lemma 3.2]{DuFa18},
so that $\mathbb{A}_{k}^{3}\setminus\{x=y=0\}$ and $X_{3}\setminus\{x=y=t=0\}$
are \'etale locally trivial $\mathbb{G}_{a}$-torsors over the same
space $\mathfrak{S}$. The Rees algebras for the induced $\mathbb{G}_{a}$-actions
on $X_{1}$ and $X_{3}$ are isomorphic to the quotients of $R(k[x,y,z,t],\Delta)$
by the homogeneous ideals generated by $w-t$ and $w-(t^{3}+x)$ respectively,
hence to 
\[
R_{1}=k[x,t][y,\upsilon,z]/(x^{2}z-y^{2}-t\upsilon^{2})\quad\textrm{and}\quad R_{3}=k[x,t][y,\upsilon,z]/(x^{2}z-y^{2}-(t^{3}+x)\upsilon^{2})
\]
respectively. Since these are not polynomial rings in two variables
over 
\[
\Gamma(\mathfrak{S},\mathcal{O}_{\mathfrak{S}})\cong k[x,t]=\Gamma(X_{1},\mathcal{O}_{X_{1}})^{\mathbb{G}_{a}}=\Gamma(X_{3},\mathcal{O}_{X_{3}})^{\mathbb{G}_{a}},
\]
it follows that the restrictions $V_{1}$ and $V_{3}$ of $V$ to
$S_{1}\cong\mathfrak{S}$ and $S_{3}\cong\mathfrak{S}$ are both nontrivial
vector bundles of rank $2$.
\begin{lem}
The vector bundles $p_{1}:V_{1}\rightarrow\mathfrak{S}$ and $p_{3}:V_{3}\rightarrow\mathfrak{S}$
are not isomorphic.
\end{lem}

\begin{proof}
Indeed, if $V_{1}$ and $V_{3}$ were isomorphic as vector bundles
then the graded $\Gamma(\mathfrak{S},\mathcal{O}_{\mathfrak{S}})$-algebras
\[
\Gamma(V_{1},\mathcal{O}_{V_{1}})\cong\Gamma(\mathfrak{S},p_{1*}\mathcal{O}_{V_{1}})\quad\textrm{and}\quad\Gamma(V_{3},\mathcal{O}_{V_{3}})\cong\Gamma(\mathfrak{S},p_{3*}\mathcal{O}_{V_{3}})
\]
would be isomorphic. Combined with Lemma \ref{lem:Local-vs-Global-codim2},
this would imply that $R_{1}$ and $R_{3}$ are isomorphic graded
$k[x,t]$-algebras, hence that $k[x,y,z]=\Gamma(X_{1},\mathcal{O}_{X_{1}})\cong R_{1}/(1-\upsilon)R_{1}$
and $\Gamma(X_{3},\mathcal{O}_{X_{3}})\cong R_{3}/(1-\upsilon)R_{3}$
were isomorphic, contradicting the fact the Russel cubic $X_{3}$
is not isomorphic to $\mathbb{A}_{k}^{3}$ \cite{KaML97}.
\end{proof}

\subsection{Winkelmann\textquoteright s proper locally trivial action on $\mathbb{A}_{k}^{5}$}

A $\mathbb{G}_{a}$-action on an affine space $\mathbb{A}_{k}^{n}$
is called a translation if its geometric quotient $\mathbb{A}_{k}^{n}/\mathbb{G}_{a}$
is isomorphic to $\mathbb{A}_{k}^{n-1}$ and $\mathbb{A}_{k}^{n}$
is equivariantly isomorphic to $(\mathbb{A}_{k}^{n}/\mathbb{G}_{a})\times\mathbb{G}_{a}$
on which $\mathbb{G}_{a}$ acts by translations on the second factor.
It is classical that proper $\mathbb{G}_{a}$-actions on $\mathbb{A}_{k}^{2}$
and $\mathbb{A}_{k}^{3}$ are translations. The question whether a
proper $\mathbb{G}_{a}$-action on $\mathbb{A}_{k}^{4}$ is a translation
is still widely open (see e.g. \cite{DFJ14,Ka18} for partial results).
Examples of proper $\mathbb{G}_{a}$-actions on affine spaces $\mathbb{A}_{k}^{n}$,
$n\geq5$, which fail to be translations were constructed by Winkelmann
\cite{Win90}. We consider the simplest of these examples, in dimension
$5$.

Let $A=k[u,v][x,y,z]$ and let $\mu$ be the fixed point free $\mathbb{G}_{a}$-action
on $\mathbb{A}_{k}^{5}=\mathrm{Spec}(A)$ determined by the triangular
locally nilpotent $k[u,v]$-derivation 
\[
\partial=u\tfrac{\partial}{\partial x}+v\tfrac{\partial}{\partial y}+(1+w)\tfrac{\partial}{\partial z},
\]
where $w=xv-yu$.
\begin{lem}
\label{lem:Rees-WinkelA5}The Rees algebra $R(A,\partial)$ is isomorphic
to the quotient of the polynomial ring $A[w,c_{1},c_{2},\upsilon]$
endowed with the grading defined by the weights $\omega_{u}=\omega_{v}=\omega_{w}=\omega_{c_{1}}=\omega_{c_{2}}=0$,
$\omega_{x}=\omega_{y}=\omega_{z}=\omega_{\upsilon}=1$ by the homogeneous
ideal $I$ generated by the polynomials $vc_{1}-uc_{2}-w(w+1)$, $c_{2}x-c_{1}y+wz$,
$w\upsilon+uy-vx$, $c_{2}\upsilon+vz-(1+w)y$ and $c_{1}\upsilon+uz-(1+w)x$.
\end{lem}

\begin{proof}
By \cite{Win90}, the kernel $A_{0}=F_{0}$ of $\partial$ is generated
by $u$, $v$, $w=xv-yu$, $c_{1}=x(1+w)-uz$ and $c_{2}=y(1+w)-vz$,
with the unique relation\textcolor{blue}{{} }$vc_{1}-uc_{2}-w(1+w)=0$.
Furthermore, since we have 
\[
F_{1}=F_{0}\cdot x+F_{0}\cdot y+F_{0}\cdot z+F_{0},
\]
the Rees algebra $R(A,\partial)$ is generated over $A_{0}$ by $x$,
$y$, $z$ and $\upsilon$. These elements of $A$ satisfy the linear
dependence relations $c_{1}\upsilon+uz-(1+w)x=0$, $c_{2}\upsilon+vz-(1+w)y=0$
and $w\upsilon+uy-vx=0$ over $A_{0}$. Furthermore, the element $c_{2}x-c_{1}y+wz\in F_{1}$
belongs to $F_{0}$, which yields the additional relation $c_{2}x-c_{1}y+wz=0\cdot\upsilon=0$.
It follows that $R(A,\partial)$ is a quotient of the ring $B=k[u,v,w,c_{1},c_{2}][x,y,z,\upsilon]/I$.
The images of $w$ and $(1+w)$ generate the unit ideal in $B$, and
the localizations $B_{(1+w)}\cong A_{0}[(1+w)^{-1}][z,\upsilon]$
and $B_{w}\cong A_{0}[w^{-1}][x,y]$ are integral domains of dimension
$6$. The homomorphism $B\rightarrow B_{w}\times B_{(1+w)}$ is thus
injective, which implies in turn that $B$ is an integral domain of
dimension $6$. Since $R(A,\partial)$ is itself an integral domain
of dimension $6$ as $A$ is of dimension $5$, we conclude that $R(A,\partial)=B$.
\end{proof}
The image of the algebraic quotient morphism $\pi:\mathbb{A}_{k}^{5}\rightarrow Q=\mathrm{Spec}(A_{0})$
is equal to the complement of the codimension $2$ closed subset $W=\{u=v=1+w=0\}\cong\mathrm{Spec}(k[c_{1},c_{2}])$
in the smooth affine quadric 
\[
Q=\{vc_{1}-uc_{2}-w(1+w)=0\}\subset\mathbb{A}_{k}^{5}.
\]
Furthermore, the corestriction $\rho:\mathbb{A}_{k}^{5}\rightarrow S=Q\setminus W$
of $\pi$ is a $\mathbb{G}_{a}$-torsor, whose class in $H^{1}(S,\mathcal{O}_{S})$
is represented by the \v{C}ech $1$-cocycle 
\[
\left\{ \frac{c_{1}}{u(1+w)},-\frac{c_{2}}{v(1+w)},\frac{w}{uv}\right\} \in C^{1}(\mathcal{U},\mathcal{O}_{S})
\]
on the covering $\mathcal{U}$ of $S$ by the principal affine open
subsets $S_{u}$, $S_{v}$ and $S_{1+w}$. Viewing $\mathbb{A}_{k}^{5}$
as a $\mathbb{G}_{a}$-torsor over $S$, it follows from Proposition
\ref{prop:Ga-torsors-Rees} that the quasi-coherent $\mathcal{O}_{S}$-algebra
$\mathcal{R}(\mathbb{A}_{k}^{5},\mu)$ is isomorphic to the symmetric
algebra $\mathrm{Sym}^{\cdot}\mathcal{F}_{1}$ of the rank $2$ locally
free sheaf $\mathcal{F}_{1}=\mathcal{K}er\delta_{S}^{2}$, where $\delta_{S}$
denotes the $\mathcal{O}_{S}$-derivation of $\rho_{*}\mathcal{O}_{\mathbb{A}_{k}^{5}}$
induced by $\partial$. Let $p:V=\mathrm{Spec}_{S}(\mathrm{Sym}^{\cdot}\mathcal{F}_{1})\rightarrow S$
be the corresponding vector bundle.

It follows from the proof of Lemma \ref{lem:Rees-WinkelA5} that the
morphism $\overline{p}:Y=\mathrm{Spec}(R(A,\partial))\rightarrow Q$
induced by the inclusion $A_{0}=\Gamma(Q,\mathcal{O}_{Q})\subset R(A,\partial)$
is vector bundle of rank $2$, which becomes trivial on the cover
of $Q$ by the the principal affine open subsets $Q_{w}$ and $Q_{(1+w)}$,
and whose restriction over $S=Q\setminus W$ coincides with the vector
bundle $p:V=\mathrm{Spec}_{S}(\mathrm{Sym}^{\cdot}\mathcal{F}_{1})\rightarrow S$.
 
\begin{lem}
The vector bundles $\overline{p}:Y=\mathrm{Spec}(R(A,\partial))\rightarrow Q$
and $p:V=\mathrm{Spec}_{S}(\mathrm{Sym}^{\cdot}\mathcal{F}_{1})\rightarrow S$
are nontrivial.
\end{lem}

\begin{proof}
The homomorphism $A_{0}[x,y,z,\upsilon]\rightarrow A_{0}$, $(x,y,z,\upsilon)\mapsto(u,v,(1+w),0)$
induces a homomorphism $R(A,\partial)\rightarrow A_{0}$ defining
a section $s:Q\rightarrow Y$ of $\overline{p}:Y\rightarrow Q$ whose
zero locus is equal to the closed variety $W$. Since $W$ is not
a scheme-theoretic complete intersection in $Q$ (see e.g. \cite[Lemma 6.3]{Sri08}),
it follows that $\overline{p}:Y\rightarrow Q$ is a nontrivial vector
bundle. In particular $R(A,\partial)$ is not isomorphic to a polynomial
ring in two variables over $A_{0}$.  This implies that $p:V\rightarrow S$
is a nontrivial vector bundle. Indeed, otherwise, since $W$ has pure
codimension $2$ in the smooth affine variety $Q$, $V$ would extend
to the trivial vector bundle on $Q$, and then $\Gamma(V,\mathcal{O}_{V})=R(A,\partial)$
would be isomorphic to a polynomial ring in two variables over $A_{0}$. 
\end{proof}

\subsection{Extensions of $\mathbb{G}_{a}$-torsors over punctured surfaces}

In this subsection, we present an application of Rees algebras to
the construction of affine extensions of $\mathbb{G}_{a}$-torsors
over punctured surfaces. The following notion was introduced in \cite{DHK,He15}.
\begin{defn}
Let $(S',o)$ be a pair consisting of the spectrum of a regular local
ring $B$ essentially of finite type and dimension $2$ over an algebraically
closed field $k$ of characteristic zero and its closed point. A \emph{normal
affine} $\mathbb{G}_{a}$-\emph{extension} of a nontrivial $\mathbb{G}_{a}$-torsor
$\rho:P\rightarrow S=S'\setminus\{o\}$ is a $\mathbb{G}_{a}$-equivariant
open embedding $j:P\hookrightarrow X$ into an integral normal $k$-scheme
$X$ equipped with a surjective affine morphism $\pi:X\rightarrow S'$
of finite type and a $\mathbb{G}_{a,S}$-action, such that the commutative
diagram \[\begin{tikzcd}    P \arrow[r,hook,"j"] \arrow[d,"\rho"'] & X \arrow[d,"\pi"] \\ S=S'\setminus\{o\} \arrow[r,hook,"h"] & S' \end{tikzcd}\]
is cartesian, where $h:S\rightarrow S'$ denotes the open inclusion.
\end{defn}

Let $\mu_{P}:\mathbb{G}_{a,S}\times_{S}P\rightarrow P$ be the $\mathbb{G}_{a,S}$-action
on $P$ with corresponding locally nilpotent $\mathcal{O}_{S}$-derivation
$\partial_{P}$ of $\rho_{*}\mathcal{O}_{P}$. It follows from Example
\ref{exa:Relative-Rees-SL2} (see also Example \ref{subsec:Rees-algebra-SL2})
that the Rees algebra $\mathcal{R}(P,\mu)=\bigoplus_{n\geq0}\mathcal{F}_{n}$
is isomorphic to the polynomial ring algebra $\mathcal{O}_{S}[u,v]$
where $u$ and $v$ are homogeneous variables of degree $1$. Furthermore,
by Proposition \ref{prop:Ga-torsors-Rees} a), the associated homogeneous
$\mathcal{O}_{S}$-derivation $\mathcal{R}(\partial_{P})$ of degree
$-1$ of $\mathcal{R}(P,\mu)$ is equal to $x\frac{\partial}{\partial u}+y\frac{\partial}{\partial v}$
for some $x,y\in\Gamma(S,\mathcal{O}_{S})=B$ such that $(x,y)\mathcal{O}_{S}=\mathcal{O}_{S}$.
By Example \ref{exa:Relative-Rees-SL2}, the nontriviality of $\rho:P\rightarrow S$
is equivalent to the property that the radical of $(x,y)B$ is equal
to the maximal ideal $\mathfrak{m}$ of $B$. The image of the element
$\upsilon\in\mathcal{F}_{1}$ in $\mathcal{O}_{S}[u,v]$ is then equal
to $xv-yu$, and $P$ is isomorphic to the closed subscheme of $\mathrm{Spec}(B[u,v])$
defined by the equation $\upsilon=1$. 

Given a normal affine extension $j:P\hookrightarrow X=\mathrm{Spec}(A)$,
where $A$ is a $B$-algebra of finite type, the $\mathbb{G}_{a,S}$-action
on $X$ is determined by a locally nilpotent $B$-derivation $\partial_{X}$
of $A$. We denote by $\{F_{X,n}\}_{n\geq0}$ the corresponding filtration
of $A$ by its $B$-submodules and by $R(A,\partial_{X})=\bigoplus_{n\geq0}F_{X,n}$
the associated Rees $B$-algebra.
\begin{prop}
There is one-to-one correspondence between:

1) Normal affine $\mathbb{G}_{a}$-extensions $j:P\hookrightarrow X=\mathrm{Spec}(A)$
whose Rees algebras $R(A,\partial_{X})$ are finitely generated over
$B$.

2) Normal finitely generated proper graded $B$-subalgebras $R_{X}=\bigoplus_{n\geq0}G_{n}$
of $B[u,v]$ which are stable under the derivation $x\frac{\partial}{\partial u}+y\frac{\partial}{\partial u}$,
containing $\upsilon=xv-yu$ and such that $\tilde{G}_{1}|_{S}=\mathcal{F}_{1}$,
where $\tilde{G}_{1}$ denotes the sheaf of $\mathcal{O}_{S'}$-modules
associated to the $B$-module $G_{1}$.
\end{prop}

\begin{proof}
Given a normal affine extension $j:P\hookrightarrow X$ of $\rho:P\rightarrow S$,
the commutativity of the diagram in the definition implies that we
have an injective homomorphism 
\[
j^{*}:\pi_{*}\mathcal{O}_{X}\rightarrow\pi_{*}j_{*}\mathcal{O}_{P}=h_{*}\rho_{*}\mathcal{O}_{P}.
\]
Let $\mu_{X}:\mathbb{G}_{a,S'}\times_{S'}X\rightarrow X$ be the $\mathbb{G}_{a,S'}$-action
on $X$ and let $\{\mathcal{F}_{X,n}\}_{n\geq0}$ be the corresponding
ascending filtration of the $\mathcal{O}_{S'}$-algebra $\mathcal{A}=\pi_{*}\mathcal{O}_{X}$.
The open embedding $j$ being by definition $\mathbb{G}_{a}$-equivariant
with respect to the morphism $\tilde{h}:\mathbb{G}_{a,S}\rightarrow\mathbb{G}_{a,S'}$
induced by the open inclusion $h:S\rightarrow S'$, it follows from
Proposition \ref{prop:EquivariantMor-Rees} that $j^{*}$ induces
an injective homomorphism of graded $\mathcal{O}_{S'}$-algebras
\[
\mathcal{R}(j^{*}):\mathcal{R}(X,\mu_{X})=\bigoplus_{n\geq0}\mathcal{F}_{X,n}\rightarrow h_{*}\mathcal{R}(P,\mu_{P})=\bigoplus_{n\geq0}h_{*}\mathcal{F}_{n}\cong h_{*}\mathcal{O}_{S}[u,v]=\mathcal{O}_{S'}[u,v]
\]
which is equivariant with respect to the associated homogeneous $\mathcal{O}_{S'}$-derivations
$\mathcal{R}(\partial_{X})$ and $h_{*}\mathcal{R}(\partial_{P})$.
Furthermore, since by definition of an affine extension the open embedding
$j$ restricts to an equivariant isomomorphism over $S$, it follows
that $\mathcal{R}(j^{*})$ restricts to an equivariant isomorphism
over $S$. Taking global sections over $S'$, we obtain an injective
homomorphism of graded $B$-algebras 
\[
\Gamma(S',\mathcal{R}(j^{*})):\Gamma(S',\mathcal{R}(X,\mu_{X}))\rightarrow\Gamma(S',\mathcal{O}_{S'}[u,v])=B[u,v],
\]
which is equivariant with respect to the locally nilpotent $B$-derivations
$\mathcal{R}(\partial_{X})$ and $\Gamma(h_{*}\mathcal{R}(\partial_{P}))=x\frac{\partial}{\partial u}+y\frac{\partial}{\partial v}$.
Since $o$ has codimension 2 in the regular scheme $S'$, we have
$\Gamma(S',\mathcal{R}(X,\mu_{X}))=R(A,\partial_{X})$ and $\mathcal{R}(X,\mu_{X})=\widetilde{R(A,\partial_{X})}$.
Since $A$ is normal, so is $R(A,\partial_{X})$ by Lemma \ref{lem:Normalization-Commute-Rees}.
Since $\mathcal{R}(j^{*})$ maps the constant section $1\in\Gamma(S',\mathcal{O}_{S'})$
viewed in $\mathcal{F}_{X,1}$ to the same section viewed in $h_{*}\mathcal{F}_{1}$,
it follows from the definition of $\upsilon$ (see (\ref{eq:Graded-inclusions}))
that $xv-yu\in R(A,\partial_{X})$. If the inclusion $R(A,\partial_{X})\subset B[u,v]$
is an equality, then 
\[
A\cong R(A,\partial_{X})/(1-\upsilon)R(A,\partial_{X})=B[u,v]/(1-\upsilon)B[u,v]\cong\Gamma(P,\mathcal{O}_{P}),
\]
which contradicts the fact that $j:P\hookrightarrow X$ is an affine
extension. Since by Proposition \ref{prop:Ga-torsors-Rees} b), $\mathcal{R}(P,\mu_{P})\cong\mathrm{Sym}^{\cdot}\mathcal{F}_{1}$,
the equality $\widetilde{R(A,\partial_{X})}|_{S}=\mathcal{R}(P,\mu_{P})$
is equivalent to the fact that $\widetilde{F}_{X,1}|_{S}=\mathcal{F}_{X,1}|_{S}=\mathcal{F}_{1}$.
Summing up, independently of whether $R(A,\partial_{X})$ is finitely
generated over $B$ or not, $R(A,\partial_{X})\subset B[u,v]$ is
an integrally closed proper graded $B$-subalgebra of $B[u,v]$, stable
under the derivation $x\frac{\partial}{\partial u}+y\frac{\partial}{\partial u}$,
containing $\upsilon=xv-yu$ and such that $\widetilde{F}_{X,1}|_{S}$.

Conversely, given a finitely generated $B$-subalgebra $R_{X}=\bigoplus_{n\geq0}G_{n}$
of $B[u,v]$ satisfying all these properties, the quotient $A=R_{X}/(1-\upsilon)R_{X}$
is a finitely generated $B$-subalgebra of $B[u,v]/(1-\upsilon)B[u,v]\cong\Gamma(P,\mathcal{O}_{P})$,
stable under the derivation $\partial_{P}$ and such that for the
induced locally nilpotent $B$-derivation $\partial_{X}=\partial_{P}|_{A}$,
we have $R(A,\partial_{X})\cong R_{X}$. By Lemma \ref{lem:upsilon-inversion},
we have $R_{X}[\upsilon^{-1}]\cong A[\upsilon^{\pm1}]$, so that $A$
is normal as $R_{X}$ is normal by assumption. The affine $S'$-scheme
$\pi:X=\mathrm{Spec}(A)\rightarrow S'$ is thus normal and of finite
type, and the morphism $j:P\rightarrow X$ is equivariant. Since $\tilde{G}_{1}|_{S}=\mathcal{F}_{1}$,
it follows that the restriction of $j$ over $S$ is an isomorphism,
so that $j$ is an open embedding of $P$ with complement equal to
$\pi^{-1}(o)$. Finally, the inclusion $A\subset\Gamma(P,\mathcal{O}_{P})$
is strict since otherwise we would have $R_{X}=B[u,v]$. It follows
that $\pi^{-1}(o)$ is not empty, hence that $j:P\hookrightarrow X$
is a normal affine extension of $P$ with finitely generated Rees
$B$-algebra $R(A,\partial_{X})\cong R_{X}$.
\end{proof}
\begin{example}
The graded proper $B$-subalgebra 
\[
R=B[\upsilon,xu,xv,yv]=\bigoplus_{n\geq0}G_{n}
\]
of $B[u,v]$ is generated in degree $1$, stable under the derivation
$R(\partial_{P})=x\frac{\partial}{\partial u}+y\frac{\partial}{\partial v}$
and satisfies $\tilde{G}_{1}|_{S}\cong\mathcal{O}_{S}^{\oplus2}=\mathcal{F}_{1}|_{S}$.
Writing $X=xu$, $Z=xv$ and $Y=yv$, we have $R\cong B[\upsilon,X,Z,Y]/J$
where $J$ is the homogeneous ideal generated by 
\[
xY-yZ,\,yX-x(Z-\upsilon),\,XY-Z(Z-\upsilon).
\]
We thus have 
\[
R/(1-\upsilon)R\cong B[X,Z,Y]/(xY-yZ,yX-x(Z-1),XY-Z(Z-1))
\]
which is easily seen to be smooth by the Jacobian criterion. This
implies in turn by Lemma \ref{lem:Normalization-Commute-Rees} that
$R$ is normal. The induced $\mathbb{G}_{a,S'}$-action on $V=\mathrm{Spec}(R/(1-\upsilon)R)$
is given by the locally nilpotent $B$-derivation 
\[
\partial_{V}=x^{2}\tfrac{\partial}{\partial X}+y^{2}\tfrac{\partial}{\partial Y}+xy\tfrac{\partial}{\partial Z}.
\]
The open embedding $P\hookrightarrow V$ is given by $(u,v)\mapsto(X,Y,Z)=(xu,yv,xv)$
and the fiber of $\pi:V\rightarrow S'$ over the closed point $o$
is isomorphic to the smooth surface with equation $XY-Z(Z-1)=0$ in
$\mathbb{A}_{\kappa}^{3}=\mathrm{Spec}(B/\mathfrak{m}B[X,Y,Z])$,
on which the $\mathbb{G}_{a,S'}$-action on $V$ restricts to the
trivial $\mathbb{G}_{a,\kappa}$-action.
\end{example}

\begin{example}
(See \cite[§ 3.4.1]{DHK}) For every integer $\ell\geq0$, we let
$R_{\ell}$ be the proper graded $B$-subalgebra 
\[
R_{\ell}=B[\upsilon,v,xu,uv,xu^{2},\ldots,xu^{\ell+4}]=\bigoplus_{n\geq0}G_{n}
\]
of $B[u,v]$. It is straightforward to see that $\tilde{G}_{1}|_{S}\cong\mathcal{O}_{S}^{\oplus2}=\mathcal{F}_{1}|_{S}$.
Furthermore, since $R(\partial_{P})(xu^{m})=mx^{2}u^{m-1}\in R_{\ell}$
for every $m=1,\ldots,\ell$ and $R(\partial_{P})(uv)=xv+yu=-\upsilon+2xv\in R_{\ell}$,
we get that $R_{\ell}$ is $R(\partial_{P})$-stable. The open embedding
\[
P\hookrightarrow V_{\ell}=\mathrm{Spec}(R_{\ell}/(1-\upsilon)R_{\ell})
\]
is given by $(u,v)\mapsto(v,xu,uv,xu^{2},\ldots,xu^{\ell+4})$ and
the fiber of $\pi_{\ell}:V_{\ell}\rightarrow S'$ over the closed
point $o$ of $S'$ is isomorphic to $\mathbb{A}_{\kappa}^{2}=\mathrm{Spec}(\kappa[v,uv])$,
where $\kappa=B/\mathfrak{m}B$, on which the induced $\mathbb{G}_{a,S'}$-action
on $V_{\ell}$ restricts to the free $\mathbb{G}_{a,\kappa}$-action
$t\cdot(v,uv)\mapsto(v,uv-t)$. Denoting $y\in\mathfrak{m}B$ by $y_{0}$,
$V_{\ell}$ endowed with the $\mathbb{G}_{a,S'}$-action induced by
$R(\partial_{P})$ is equivarianly isomorphic to the smooth subvariety
in $S'\times_{\mathbb{Z}}\mathbb{A}_{\mathbb{Z}}^{n+2}=\mathrm{Spec}(B[z_{1},z_{2},y_{1},\ldots,y_{n}])$
defined by the system of equations
\[
\begin{cases}
y_{i}y_{j}-y_{k}y_{\ell} & =0\quad i,j,k,\ell=0,\ldots,n,\;i+j=k+\ell\\
z_{2}y_{i}-z_{1}y_{i+1} & =0\quad i=0,\ldots,n-1\\
xy_{i+1}-y_{i}(y_{0}z_{1}+1) & =0\quad i=0,\ldots,n-1\\
xz_{2}-z_{1}(y_{0}z_{1}+1) & =0,
\end{cases}
\]
endowed with the $\mathbb{G}_{a,S'}$-action induced by the locally
nilpotent $B$-derivation 
\[
x\partial_{z_{1}}+(2y_{0}z_{1}+1)\partial_{z_{2}}+\sum_{i=1}^{n}iy_{0}y_{i-1}\partial_{y_{i}}
\]
of its coordinate ring. Since $V_{\ell}$ is smooth, hence normal,
it follows that $R_{\ell}$ is also normal by Lemma \ref{lem:Normalization-Commute-Rees}.
\end{example}

\bibliographystyle{amsplain}

\end{document}